\documentclass[a4paper,10pt%,draft
]{amsart}
\usepackage[english]{babel}
\usepackage{stmaryrd}
\usepackage{amsmath,amssymb,amsthm}
\usepackage[alphabetic, abbrev]{amsrefs}
\usepackage{enumerate}
\usepackage{mathtools, comment}
\usepackage[bookmarksopen]{hyperref}
\usepackage[all]{xy}
\newdir{ >}{{}*!/-5pt/@{>}} %% cf xyguide exercise 14
\SelectTips{cm}{} %% cf xyguide page 9

\numberwithin{equation}{section} 
\let\realequation\equation
\def\equation{\setcounter{equation}{\arabic{subsection}}%
   \refstepcounter{subsection}%
   \realequation}
\let\realmultline\multline
\def\multline{\setcounter{equation}{\arabic{subsection}}%
   \refstepcounter{subsection}%
   \realmultline}

\newtheorem{theorem}[subsection]{Theorem}
\newtheorem*{mthm}{Main Theorem}
\newtheorem{corollary}[subsection]{Corollary}
\newtheorem{lemma}[subsection]{Lemma}
\newtheorem{proposition}[subsection]{Proposition}
\theoremstyle{definition}
\newtheorem{definition}[subsection]{Definition}

\newtheorem{remark}[subsection]{Remark}

\newtheorem{example}[subsection]{Example}

\newtheorem{construction}[subsection]{Construction}
\newtheorem{observation}[subsection]{Observation}

\newcommand{\bS}{\mathbb{S}}
\newcommand{\bZ}{\mathbb{Z}}

\newcommand{\cA}{\mathcal{A}}

\newcommand{\cC}{\mathcal{C}}

\newcommand{\cI}{\mathcal{I}}

\newcommand{\cK}{\mathcal{K}}

\newcommand{\cN}{\mathcal{N}}
\newcommand{\cR}{\mathcal{R}}
\newcommand{\cS}{\mathcal{S}}

\newcommand{\ucI}{*}

\DeclareMathOperator{\hocolim}{hocolim}

\DeclareMathOperator{\colim}{colim}
\DeclareMathOperator{\const}{const}

\DeclareMathOperator{\Sing}{Sing}
\newcommand{\ot}{\leftarrow}

\newcommand{\taut}{J}
\newcommand{\tautko}{\mathrm{taut}_{\mathrm{KO}}}
\newcommand{\tautku}{\mathrm{taut}_{\mathrm{KU}}}
\newcommand{\ttt}{\natural} 
\newcommand{\op}{{\mathrm{op}}}
\newcommand{\id}{{\mathrm{id}}}
\newcommand{\tensor}{\otimes}
\newcommand{\sm}{\wedge}
\newcommand{\barsm}{\barwedge}
\newcommand{\iso}{\cong}
\newcommand{\concat}{\sqcup}
\newcommand{\bld}[1]{{\mathbf{#1}}}
\newcommand{\Spsym}[1]{{\mathrm{Sp}^{\Sigma}_{#1}}}
\newcommand{\sset}{\mathrm{sSet}}
\newcommand{\tp}{\mathrm{Top}}
\newcommand{\SpsymR}{{\mathrm{Sp}^{\Sigma}_{\cR}}}

\DeclareMathOperator{\capitalGL}{GL}

\newcommand{\GLoneIof}[1]{\capitalGL^{\cI}_1\!{#1}}
\newcommand{\GLoneof}[1]{\capitalGL_1\!{#1}}
\newcommand{\BGLoneIof}[1]{B^{\boxtimes}(\capitalGL^{\cI}_1\!{#1})}
\newcommand{\BGLoneIcofof}[1]{B^{\boxtimes}((\capitalGL^{\cI}_1\!{#1})^{\mathrm{cof}})}
\newcommand{\BGLoneof}[1]{B\!\capitalGL_1\!{#1}}

\usepackage[pdftex]{graphics,color}

\title[Homotopical and operator algebraic twisted \texorpdfstring{$K$}{K}-theory]{Homotopical and operator algebraic\\ twisted \texorpdfstring{$K$}{K}-theory}

\author{Fabian Hebestreit} \address{Mathematical Institute, University
  of Bonn, Germany} \email{f.hebestreit@math.uni-bonn.de}

\author{Steffen Sagave} \address{IMAPP, Radboud University Nijmegen, The Netherlands} \email{s.sagave@math.ru.nl}

\date{\today}

\begin{document}

\begin{abstract}
  Using the framework for multiplicative parametrized homotopy theory
  introduced in joint work with C. Schlichtkrull, we produce a
  multiplicative comparison between the homotopical and operator
  algebraic constructions of twisted $K$-theory, both in the real and
  complex case. We also improve several comparison results about
  twisted $K$-theory of $C^*$-algebras to include multiplicative
  structures. Our results can also be interpreted in the
  $\infty$-categorical setup for parametrized spectra.
\end{abstract}

\keywords{parametrized spectrum, twisted $K$-theory}
\subjclass[2010]{55P43; 55P42, 19L50}

\maketitle
\setcounter{tocdepth}{1} 

\tableofcontents
\section{Introduction}
Twisted homology and cohomology, or (co)homology with local coefficients, was originally invented by Steenrod \cite{Steenrod}. It was designed to capture cohomological invariants that can only be made sense of in ordinary (co)homology under orientation assumptions. Examples include obstruction classes against sections in fiber bundles, fundamental classes of manifolds and Thom classes of vector bundles. Donovan and Karoubi~\cite{Donovan} generalized this to obtain a twisted version of topological $K$-theory. 

In integral (co)homology, a twisting datum is given by a line bundle and thus determined by its first Stiefel--Whitney class or, equivalently, by a map to $K(\bZ/2,1) \simeq \tau_{\leq 1}\mathrm{BO} \simeq \mathrm B(\mathrm{O}/\mathrm{SO})$.  Donovan and Karoubi described twists in $K$-theory using principal $\mathrm P \mathcal U$-bundles, where $\mathrm P \mathcal U$ denotes the projective unitary group of some (infinite dimensional) Hilbert space. By Kuiper's theorem $\mathcal U$ is contractible, and hence principal $\mathrm P \mathcal U$-bundles are classified by their Dixmier--Douady classes in the third integral cohomology. Equivalently, they are encoded by maps to $K(\bZ,3) \simeq \mathrm B(\mathrm{SO} \sslash \mathrm{Spin}^c)$, where $\sslash$ denotes homotopy quotients. The theory can easily be extended to graded Hilbert spaces and to maps into $\mathrm B(\mathrm O \sslash \mathrm{Spin}^c)$. Moreover, there is an analogous theory in the case of real $K$-theory for principal $\mathrm P \mathcal O$-bundles and maps into $K(\bZ/2,2) \simeq \tau_{\leq 2}\mathrm{BSO} \simeq \mathrm B(\mathrm{SO} \sslash \mathrm{Spin})$ and more generally $\tau_{\leq 2}\mathrm{BO} \simeq \mathrm B(\mathrm O \sslash \mathrm{Spin})$ for graded bundles.

While the construction of Donovan and Karoubi relies on the operator algebraic framework for $K$-theory, there are also purely homotopical constructions of twisted $K$-theory. As a motivating example for the relevance of the interplay between these approaches we take the following conjecture of Stolz: By design, any closed $d$-manifold $M$ determines a fundamental class $[M] \in \mathrm{ko}_d(M,\mathrm{or}(M))$ in its twisted, connective $\mathrm{KO}$-homology, where $\mathrm{or}(M) \colon M \rightarrow \tau_{\leq 2}\mathrm{BO}$ records the first two (normal) Stiefel--Whitney classes of $M$. If the universal cover of $M$ admits a spin structure, then the twist $\mathrm{or}(M)$ factors through the classifying map $c \colon M \rightarrow \mathrm{B}\pi_1(M)$. Stolz conjectured that the vanishing of $c_*[M] \in \mathrm{ko}_d(\mathrm{B}\pi_1(M),\mathrm{or}(M))$ is sufficient for the existence of a metric of positive scalar curvature on $M$. This conjecture is intimately tied to the Gromov--Lawson--Rosenberg conjecture. The latter asserts that the vanishing of a certain index class $\mathrm{ind}(M) \in \mathrm{KO}_d(C^*(\pi_1(M),\mathrm{or}(M)))$, which is a well-known necessary condition for the existence of such a metric, is also sufficient (see \cite{Stolz-conc}). The connection between the two conjectures is established via an assembly map $\mathrm{ko}_*(\mathrm{B}G,\tau) \rightarrow \mathrm{KO}_*(\mathrm{B}G,\tau) \rightarrow \mathrm{KO}_*(C^*(G,\tau))$, which takes the class $c_*[M]$ to $\mathrm{ind}(M)$.  This relation therefore relies crucially on the operator algebraic interpretation of $K$-homology and assembly (see e.g. the survey \cite{RoSt} and the references therein). In contrast, one reason why the conjecture of Stolz is useful is that the groups $\mathrm{ko}_*(\mathrm{B}G,\tau)$ are more amenable to computations with the best available tools coming from homotopy theory. In particular, Stolz' conjecture and many cases of the Gromov-Lawson-Rosenberg conjecture are known in the case of vanishing twists, i.e., when $M$ itself admits a spin structure, with largely homotopical proofs.

Let us now outline the homotopical approach to twisted (co)homology theories, and, in particular, twisted $K$-theory. In general, one first builds a topological monoid $\GLoneof{R}$ from a structured ring spectrum $R$. This monoid governs the twists that the (co)homology theory represented by $R$ allows. Its classifying space $\BGLoneof{R}$ comes equipped with a universal parametrized spectrum~$\gamma_R$. Such parametrized spectra give rise to twisted (co)homology theories, and this machinery applied to the real and complex $K$-spectra $\mathrm{KO}$ and $\mathrm{KU}$ provides the homotopy theorist's version of twisted $K$-theory. To simplify the exposition, let us restrict attention to $\mathrm{KO}$. In this case, the space $\BGLoneof{\mathrm{KO}}$ is equivalent to the classifying space of $\{\pm 1\} \times \mathrm{BO}$, a monoid under the tensor product of virtual vector bundles. The construction of Donovan and Karoubi provides a map $\kappa \colon \tau_{\leq 2}\mathrm{BO} \rightarrow \BGLoneof{\mathrm{KO}}$, which allows to interpret any principal $\mathrm P \mathcal O$-bundle as a twist in this more general sense. To compare the operator algebraic and homotopical twisted $K$-theory, it therefore becomes necessary to

\begin{enumerate}[(1)] 
\item determine the map $\kappa \colon \tau_{\leq 2}\mathrm{BO} \rightarrow \BGLoneof{\mathrm{KO}}$ induced by the construction of Donovan and Karoubi in homotopical terms and 
\item identify the twisted cohomology theory  $\kappa^*(\gamma_{\mathrm{KO}})$ associated to the parametrized spectrum with the operator algebraic construction. 
\end{enumerate}

Importing the work of Donovan and Karoubi into modern homotopy theory, one can define a parametrized spectrum $\mathrm{KO} \sslash \mathrm P \mathcal O$ over $\tau_{\leq 2}\mathrm{BO} \simeq \mathrm{BP}\mathcal O$ that represents the operator algebraic version of twisted $K$-theory. Task (2) therefore boils down to showing that $\kappa^* (\gamma_\mathrm{KO})$ coincides with $\mathrm{KO} \sslash \mathrm P \mathcal O$. Concerning task (1), there is well known homotopical candidate for $\kappa$: Employing general constructions for Thom spectra, the $J$-homomorphism $J \colon \mathrm{BO} \rightarrow \BGLoneof{\bS}$ induces a map $\taut_{\mathrm{Spin}} \colon \tau_{\leq 2}\mathrm{BO} \simeq \mathrm{B}(\mathrm{O} \sslash \mathrm{Spin}) \rightarrow \BGLoneof{\mathrm{MSpin}}$ via the unit $\bS \rightarrow \mathrm{MSpin}$.  Its composite with the Atiyah--Bott--Shapiro orientation $\alpha \colon \mathrm{MSpin} \rightarrow \mathrm{KO}$ is a map $\tautko = \BGLoneof{(\alpha)} \circ \taut_{\mathrm{Spin}}$ from $\tau_{\leq 2}\mathrm{BO}$ to $\BGLoneof{\mathrm{KO}}$.

It was recently proven in \cite{AGG-Uniqueness} that when restricted to $\tau_{\leq 2}\mathrm{BSO}$, the map  $\kappa$ from the construction of Donovan and Karoubi and $\tautko$ do indeed agree. This was done by simply classifying all maps $\tau_{\leq 2}\mathrm{BSO} \simeq K(\bZ/2,2) \rightarrow \BGLoneof{\mathrm{KO}}$ up to homotopy, of which there turn out to be only two homotopy classes. However, these methods do not obviously extend to cover the more general source. More importantly, they also do not suffice to identify them as maps that preserve the multiplicative structure we explain next. Given a ring spectrum with an $E_{\infty}$ structure, i.e., a homotopy coherent commutative multiplication, the space $\BGLoneof{R}$ is an $E_{\infty}$ space and $\gamma_R$ is an $E_\infty$ object in parametrized ring spectra. These structures become important when considering for example multiplications on twisted cohomology. They also govern the theory of power operations, an important tool in computations. Both $\kappa$ and $\tautko$ can be refined to $E_{\infty}$ maps. In \cite[Appendix C]{HeJo}, Joachim and the first author sketched an identification of $\kappa$ and $\tautko$ as $E_{\infty}$ maps. However, the setup of that paper did not  allow for a multiplicative comparison of the resulting parametrized spectra.

In the present paper, we use foundational joint work with Schlichtkrull~\cite{HSS-retractive} to provide such a comparison. The main feature of the framework established in~\cite{HSS-retractive} is a point set level implementation of the convolution smash product of parametrized spectra that was previously only available in $\infty$-categorical contexts. By considering natural refinements of $\tautko^*(\gamma_\mathrm{KO})$, $\kappa^*(\gamma_\mathrm{KO})$ and  $\mathrm{KO} \sslash \mathrm P \mathcal O$ to commutative parametrized ring spectra in this setting, we show:

\begin{mthm}
  There is a preferred equivalence in the homotopy category of
  commutative parametrized ring spectra relating the refinements of
  $\mathrm{KO} \sslash \mathrm P \mathcal O$ and
  $\tautko^*(\gamma_\mathrm{KO})$. Consequently, operator algebraic
  and homotopical twisted $K$-theory agree including their
  multiplicative structure.
\end{mthm}
We will give a more precise formulation in Theorem \ref{thm:d} below,
after reviewing the setup of \cite{HSS-retractive}.
 
The base spaces of the parametrized ring spectra appearing in the theorem both come with equivalences to the $E_{\infty}$ space $\tau_{\leq 2}\mathrm{BO}$ obtained by truncating the Whitney sum $E_\infty$ structure on $\mathrm{BO}$. When passing to the base spaces, the equivalence in the theorem specializes to the identification of $\kappa$ and $\tautko$ as $E_{\infty}$ maps discussed above.

The extensive comparison results of~\cite[Section~10]{HSS-retractive} allow us to obtain the statement of the theorem also in the $\infty$-categorical models for parametrized spectra considered in \cite{Ando-B-G-H-R_infinity-Thom,Ando-B-G_parametrized}. To an $E_\infty$-space the authors of these papers associate a symmetric monoidal $\infty$-category of parametrized spectra, which is equivalent to the underlying $\infty$-category of the category of parametrized spectra with the convolution smash product employed in the present paper. Furthermore, under this equivalence commutative parametrized ring spectra such as $\tautko^*(\gamma_{\mathrm{KO}})$ correspond to the $E_\infty$-parametrized ring spectra of the same name. Objects arising from universal constructions such as $\gamma_{\mathrm{KO}}$ and the map  $\tautko$ are particularly amenable to manipulations in the higher categorical setup, and we think of our framework as a useful bridge between the operator algebraic and the $\infty$-categorical world.

\subsection{Parametrized symmetric spectra}
We now outline the definition of the parametrized spectra that we will be using throughout the paper. These are the \emph{symmetric spectra in retractive spaces} developed in a companion paper joint with Schlichtkrull~\cite{HSS-retractive}. 
 
Starting from the category of spaces $\cS$, we first pass to the category of retractive spaces $\cS_\cR$, also called ex-spaces. They consists of pairs of spaces $(E,X)$ with structure maps maps $X \rightarrow E$ and $E \rightarrow X$ that compose to the identity of $X$. We think of $E$ and $X$ as the total and base spaces, respectively. Maps of retractive spaces are allowed to vary both the total and the base space. Retractive spaces admit a model category structure with weak equivalences the maps that are weak homotopy equivalences on both the base and the total space. Furthermore, $\cS_\cR$ comes with a symmetric monoidal structure given by the \emph{fiberwise smash-product}~$\barsm$. It restricts to the cartesian product on base spaces, has $(S^0,*)$ as monoidal unit, and equips $\cS_\cR$ with the structure of a symmetric monoidal model category.

A symmetric spectrum in retractive spaces is then defined to be a sequence of retractive spaces $(E,X)_n$ together with actions of the symmetric groups $\Sigma_n$ and structure maps $  (E,X)_n\barsm (S^1,*)\rightarrow (E,X)_{n+1}$ that are suitably equivariant. We write $\SpsymR$ for the resulting category and simply refer to its objects as parametrized symmetric spectra. By results of Hovey~\cite{Hovey_symmetric-general},  $\SpsymR$ inherits both a symmetric monoidal structure, again denoted by $\barsm$, and a compatible model structure that we refer to as the \emph{local} model structure. The base spaces of a parametrized symmetric spectrum $(E,X)$ assemble into an $\cI$-space $X$, that is,  a functor $\cI \rightarrow \cS$ from the category of finite sets and injections $\cI$. Here the $\Sigma_n$-action on $X_n$ provides functoriality in the endomorphisms of $\bld{n}=\{1,\dots,n\}$, the structure maps give the functoriality in the subset inclusions $\bld{n} \hookrightarrow \bld{n+1}$, and their compatibility implies that this data extends to a functor $\cI \to \cS$.

The category of $\cI$-spaces $\cS^\cI$ is also equipped with a symmetric monoidal Day convolution product $\boxtimes$. It has a compatible model structure with weak equivalences the $\cI$-equivalences, i.e., the maps that induce weak homotopy equivalences when applying homotopy colimits~\cite[\S{}3]{Sagave-S_diagram}. The projection to the base $\pi_b \colon \SpsymR \rightarrow \cS^\cI$ is strong symmetric monoidal for the two products and both left and right Quillen. In particular, it carries local equivalences to $\cI$-equivalences. For an $\cI$-space $X$ one can consider the category $\Spsym{X} = \pi_b^{-1}(X)$ of $X$-relative symmetric spectra. It is the subcategory of $\SpsymR$ whose objects have $X$ as base $\cI$-space and whose morphisms project to the identity on~$X$. Since the base $\cI$-space may not be constant, the levels of these spectra can live in different categories. We show in~\cite[Theorem~1.2]{HSS-retractive} that $\Spsym{X}$ inherits a model category structure from $\SpsymR$ and that conversely the model structure on $\SpsymR$ can be assembled out of the various model categories $\Spsym{X}$ by a Grothendieck construction. In particular, the categories $\Spsym{X}$ are functorial in $X$ via Quillen adjunctions $f_! \colon \Spsym{X} \rightleftarrows \Spsym{Y} \colon f^*$ for maps $f \colon X \rightarrow Y$. These adjoints are important in the formation of twisted (co)homology theories. If $f$ is an $\cI$-equivalence then $(f_!, f^*)$ is a Quillen equivalence. When $X = \const_{\cI}K$ is a constant $\cI$-diagram on a space $K$, the category of $X$-relative symmetric spectra is equivalent to the category of symmetric spectrum objects in spaces over and under $K$ and thus models the usual category of parametrized spectra. 

A central feature of $\cI$-spaces is that they allow to model $E_{\infty}$ spaces by \emph{commutative $\cI$-space monoids}, i.e., by strictly commutative monoids with respect to $\boxtimes$~\cite[Theorem 1.2]{Sagave-S_diagram}. If $M$ is a commutative $\cI$-space monoid, we show that the category $\Spsym{M}$ inherits a well-behaved symmetric monoidal structure, the \emph{convolution smash product}~\cite[Section~4]{HSS-retractive}. To our knowledge, this product is not present in other point-set approaches to parametrized homotopy theory. It is essential for even stating the main results of the present paper including the multiplicative structure.

Finally, we note that the universal $R$-line bundle $\gamma_R$ considered above has a convenient point set level model in parametrized symmetric spectra. To define it, we consider a (positive fibrant) commutative symmetric ring spectrum and let $G$ be a (cofibrant replacement of) of $\GLoneIof{R}$, the commutative $\cI$-space monoid model of the units of $R$ introduced by Schlichtkrull~\cite{Schlichtkrull_units}. Writing  $\bS_t^\cI \colon \cS^\cI \rightarrow \SpsymR$ for the fiberwise suspension spectrum functor with $\bS^\cI_t[X]_n = (X_n \times S^n, X_n)$, we define $\gamma_R$ to be (a fibrant replacement of) the two-sided bar construction $B^\barsm(\bS, \bS_t^\cI[G],R)$ in $\SpsymR$. The base space of $\gamma_R$ is thus the classifying space $B^{\boxtimes}(G) = B^\boxtimes(*,G,*)$ of the units of $R$, which is again a commutative $\cI$-space monoid.  By construction, $\gamma_R$ is a commutative monoid in the category $\Spsym{B^{\boxtimes}(G)}$ under the convolution smash product. It is this object which for $R = \mathrm{KO}$ appears in the main theorem.

\subsection{Actions of cartesian \texorpdfstring{$\cI$}{I}-monoids} In the present paper,  we introduce another construction principle for parametrized symmetric ring spectra that we use to define $\mathrm{KO} \sslash \mathrm P\mathcal O$. To describe it, we say that a \emph{cartesian $\cI$-monoid} is an $\cI$-diagram of simplicial or topological monoids. Equivalently, it is a monoid in $\cS^{\cI}$ with respect to the cartesian product. An action of a cartesian $\cI$-monoid $H$ on a commutative symmetric ring spectrum $R$ is a family of $H(\bld{n})$-actions on $R_n$ compatible with base points and structure maps. Special cases of such actions were already considered by Dadarlat--Pennig~\cite{DP-I} and May--Sigurdsson~\cite[\S{}23.4]{May-S_parametrized}. When $H$ acts on $R$, we can form the \emph{homotopy quotient} $R \sslash H$ with $ (R \sslash H)_n = B^\times(*, H(\bld{n}), R_n)$ where $B^{\times}$ denotes the bar construction with respect to the cartesian product. The base $\cI$-space of $R \sslash H$ is $B^{\times}(H) = B^{\times}(*,H,*)$, the bar construction of $H$ with respect to the cartesian product. 

There is a canonical map $H \boxtimes H \to H \times H$. By restriction along it, a cartesian $\cI$-monoid $H$ can also be viewed as an $\cI$-space monoid, i.e., as a monoid in $\cS^{\cI}$ with respect to the $\boxtimes$-product. We say that a cartesian $\cI$-monoid is \emph{$\cI$-commutative} if its underlying $\cI$-space monoid is commutative. This condition can also be formulated directly in terms of an Eckmann--Hilton condition on $H$, but the salient feature for now is that it is strictly weaker than requiring each $H(\bld{n})$ to be commutative. For example, the orthogonal groups $O(n)$ assemble to an $\cI$-commutative cartesian $\cI$-monoid denoted by $O_{\cI}$. As we will see later, many other examples of interest have this property, in particular the projective orthogonal groups of tensor powers of a Hilbert space. If $H$ is $\cI$-commutative and the action of $H$ on $R$ also satisfies an Eckmann--Hilton condition, then $B^{\times}(H)$ is a commutative $\cI$-space monoid and $R \sslash H$ is a commutative monoid in $\Spsym{B^\times\!(H)}$ under the convolution smash product. Moreover, in this case the action is adjoint to a map of commutative $\cI$-space monoids $\mu^{\flat} \colon H \to \Omega^{\cI}(R)$ first considered in~\cite{DP-I}. It sends $h \in H(\bld{n})$ to the composite $S^n \to R_n \to R_n$ of the unit with the action of $h$. Here $\Omega^{\cI}(R)$ is the commutative $\cI$-space monoid model for the underlying multiplicative $E_{\infty}$ space of $R$ defined on objects by $(\Omega^{\cI}R)(\bld{n}) = \mathrm{Map}(S^n,R_n)$. If the induced multiplication on the Bousfield--Kan homotopy colimit $H_{h\cI} = \hocolim_{\cI}(H)$ is in addition grouplike, we get a map of commutative $\cI$-space monoids $\mu^{\mathrm{cof}}\colon H^{\mathrm{cof}} \to G = (\GLoneIof{R})^{\mathrm{cof}}$ between cofibrant replacements of $H$ and the units of $R$. The induced map of classifying spaces $B^{\boxtimes}(\mu^{\mathrm{cof}}) \colon B^{\boxtimes}(H^{\mathrm{cof}}) \rightarrow B^{\boxtimes}(G)$ allows us to relate $R \sslash H$ to the universal bundle $R$-line bundle~$\gamma_R$. The following statement is the main technical result of the present paper.

\begin{theorem}\label{thm:RsslashH-identification}
Suppose that $H$ and its action on $R$ satisfy the above mentioned commutativity assumptions, that $H_{h\cI}$ is grouplike, and that both $H$ and $R$ are suitably fibrant. Then the comparison map $\iota \colon B^{\boxtimes}(H^{\mathrm{cof}}) \rightarrow B^\times(H)$ is an $\cI$-equivalence. Moreover, the parametrized spectra $B^{\boxtimes}(\mu^{\mathrm{cof}})^* (\gamma_R)$  and $\iota^*(R \sslash H)$ are related by a natural zig-zag of local equivalences of commutative parametrized ring spectra in $\Spsym{B^{\boxtimes}(H^{\mathrm{cof}})}$ with respect to the convolution smash product. 
\end{theorem}

\subsection{Two applications to twisted \texorpdfstring{$K$}{K}-theory}
As advertised, the first application concerns the two models of twisted $K$-theory. We briefly sketch the constructions in the present language (ignoring cofibrant or fibrant replacements). 

We first consider the homotopical setup. Let $\theta \colon H \rightarrow O_\cI$ be a map of grouplike $\cI$-commutative cartesian $\cI$-monoids into the cartesian $\cI$-monoid formed by the orthogonal groups. Then there is a tautological map
\[\taut_\theta \colon O_\cI \sslash H \longrightarrow \GLoneIof{M\theta},\] 
where $O_{\cI} \sslash H = B^\times(*,H,O_{\cI})$ is a model for the homotopy fiber of $\theta$ and $M\theta$ is the associated Thom spectrum. When $H = \mathrm{Spin}_\cI$, the latter map can be composed with the Atiyah--Bott--Shapiro orientation $\alpha \colon M\mathrm{Spin} \rightarrow \mathrm{KO}$. This provides a map 
\[\tautko \colon B^{\boxtimes}(O_\cI \sslash \mathrm{Spin}_\cI) \longrightarrow \BGLoneIof{\mathrm{KO}}.\]
In the complex case, we obtain $\tautku \colon B^{\boxtimes}(O_\cI \sslash \mathrm{Spin}^c_\cI) \to \BGLoneIof{\mathrm{KU}}$. These maps are known as the inclusion of the lower twists of $K$-theory. As described above, one can pull the bundles $\gamma_{\mathrm{KO}}$ and $\gamma_\mathrm{KU}$ back to obtain commutative parametrized ring spectra over $B^{\boxtimes}(O_\cI \sslash \mathrm{Spin}^{(c)}_\cI)$ representing the homotopical version of twisted $K$-theory. 

Now we switch to the operator algebraic setup. Most good operator algebraic models of $K$-theory spectra admit actions of projective orthogonal (or unitary) groups. We shall use the symmetric spectra $\mathrm{KO}$ and $\mathrm{KU}$ introduced by Joachim~\cite{Jo-coherence}, which also come equipped with a simple model for the Atiyah--Bott--Shapiro orientation. They admit natural Eckmann-Hilton actions by the groups $P\mathcal O_\cI$ or $P\mathcal U_\cI$ given in degree $n$ as the projective orthogonal or unitary group of $L^2(\mathbb R^n)$. This allows us to form the commutative parametrized ring spectra $\mathrm{KO} \sslash P\mathcal O_\cI$ and $\mathrm{KU} \sslash P\mathcal U_\cI$. In Proposition \ref{prop:k-cycles}, we explicitly relate the associated twisted cohomology theory to operator algebraic definitions of twisted $K$-theory using Kasparov's $\mathrm{KK}$-theory including the multiplicative structure. 

We then apply Theorem \ref{thm:RsslashH-identification} to show:

\begin{theorem}\label{thm:d}
There exist preferred zig-zags of local equivalences
\[(\tautko)^*(\gamma_\mathrm{KO}) \simeq \mathrm{KO} \sslash P\mathcal O_\cI \quad \text{and} \quad (\tautku)^*(\gamma_\mathrm{KU}) \simeq \mathrm{KU} \sslash P\mathcal U_\cI\]
of commutative parametrized ring spectra. 
\end{theorem}

The second application is a generalization of such a multiplicative identification to certain $C^*$-algebras $A$: If $\mathcal K$ denotes the compact operators on some Hilbert space, we show that the action of the automorphisms $A \otimes \mathcal K$ on the $K$-theory spectrum of $A$ can be encoded in the action of a cartesian $\cI$-monoid $\mathrm{Aut}_\cI(\mathcal K,A)$ on Joachim's model $\mathrm K_A$. When $A$ is the base field, this recovers exactly the action considered on $\mathrm{KO}$ and $\mathrm{KU}$ above. We then construct pairings
\[\big[\mathrm K_A \sslash \mathrm{Aut}_\cI(\mathcal K,A)\big] \barsm \big[\mathrm K_B \sslash \mathrm{Aut}_\cI(\mathcal K,B)\big] \longrightarrow \mathrm K_{A \otimes B} \sslash \mathrm{Aut}_\cI(\mathcal K,A \otimes B)\]
and note that these induce the usual products on associated twisted cohomology theories. 

Now, in a series of papers \cite{DP-I,DP-II,DP-III}, Dadarlat and Pennig constructed a commutative symmetric ring spectrum $K_A^\infty$ representing the homotopy type $K_A$, whenever $A$ is strongly selfabsorbing, i.e., when there exists an isomorphism ${A \otimes A} \cong A$ with particularly good properties. Similarly, the homotopy type of $\mathrm{Aut}_\cI(\mathcal K,A)$ is then represented by an $\cI$-commutative cartesian $\cI$-monoid $\mathrm{Aut}^s_{\cI}(A \otimes \cK)$ admitting an Eckmann--Hilton action on $K^\infty_A$. This gives rise to a zig-zag of local equivalences $\mathrm{K}_A\sslash \mathrm{Aut}_\cI(\mathcal K,A) \simeq \mathrm{K}^\infty_A\sslash \mathrm{Aut}^s_{\cI}(A \otimes \cK)$ that matches the exterior pairing from above with the interior one on the right hand side. Using the main theorem of \cite{DP-I} we obtain:

\begin{theorem}\label{thm:e}
For every purely infinite, strongly selfabsorbing $C^*$-algebra $A$, there is a preferred zig-zag $\gamma_{\mathrm{K}^\infty_A} \simeq \mathrm{K}^\infty_A\sslash \mathrm{Aut}^s_{\cI}(A \otimes \cK)$ of commutative parametrized ring spectra. 
In particular, for $\mathcal O_\infty$ the infinite Cuntz algebra, we obtain a preferred zigzag $\gamma_\mathrm{KU} \simeq \mathrm{K}^\infty_{\mathcal O_\infty}\sslash \mathrm{Aut}^s_{\cI}(\mathcal O_\infty \otimes \cK)$
 since $\mathrm{KU} \simeq \mathrm K_{\mathcal O_\infty}$. 
\end{theorem}

From here it is immediate that the cycle description of higher twisted $K$-theory given by Pennig in \cite{DP-II} is multiplicative.

\subsection{Organization}
In Section~\ref{sec:I-spaces}, we review $\cI$-spaces and explain how their different products are related. In Section~\ref{sec:parametrized}, we review parametrized symmetric spectra and study the actions of cartesian $\cI$-monoids. Section~\ref{sec:product-compar} provides the homotopical comparison of the products and the proof of Theorem~\ref{thm:RsslashH-identification} as Corollary~\ref{comparison corollary}. In Section~\ref{sec:examples-of-actions} we give examples for actions  of cartesian $\cI$-monoids. Section~\ref{sec:spin} provides results about actions on bordism spectra that we use in the final Section~\ref{finalsec} to prove Theorems~\ref{thm:d} and \ref{thm:e}.

\subsection{Acknowledgments}
The authors would like to thank Emanuele Dotto, Mi\-cha\-el Joachim, Irakli Patchkoria, Christian Schlichtkrull and particularly Ulrich Pennig for useful conversations and an unnamed referee for useful comments on an earlier version of this manuscript. They would also like to thank the Isaac Newton Institute for Mathematical Sciences for support and hospitality during the program ``Homotopy harnessing higher structures". Moreover, the first author held a scholarship of the German Academic Exchange Service during a year at the University of Notre Dame, when initial work on this project was undertaken. 

This work was supported by EPSRC grants EP/K032208/1 and EP/R014604/1 and by the Hausdorff Center for Mathematics at the University of Bonn, DFG GZ 2047/1, project ID 390685813.
\subsection{Notations and conventions}
We work in the categories of compactly generated weak Hausdorff spaces $\tp$ or simplicial sets $\sset$ and write $\cS$ to refer to any of these two categories of spaces as appropriate. 

\section{Review of  \texorpdfstring{$\cI$}{I}-spaces}\label{sec:I-spaces}
We begin to recall basic material about $\cI$-spaces from~\cite[\S{}3]{Sagave-S_diagram} which is central to our model for parametrized spectra. 

\begin{definition}
We write $\cI$ for the category whose objects are the finite sets $\bld{m} = \{1,\dots, m\}$, $m \geq 0$, and whose morphisms are the injections. An \emph{$\cI$-space} $X$ is a covariant space valued functor $X\colon \cI \to \cS$, and we write $\cS^{\cI}$ for the resulting functor category. A map $X \to Y$ in $\cS^{\cI}$ is an \emph{$\cI$-equivalence} if it induces a weak homotopy equivalence $X_{h\cI} \to Y_{h\cI}$ of Bousfield--Kan homotopy colimits. 
\end{definition}
The $\cI$-equivalences participate in an (absolute projective) \emph{$\cI$-model structure} on $\cS^{\cI}$ making $\cS^{\cI}$ Quillen equivalent to spaces~\cite[Theorem 3.3]{Sagave-S_diagram}. A map $X \to Y$ in $\cS^{\cI}$ is a level equivalence (resp. positive level equivalence) if $X(\bld{n})\to Y(\bld{n})$ is a weak homotopy equivalence for all $\bld{n}$ in $\cI$ (resp. all $\bld{n}$ in $\cI$ with $|\bld{n}|\geq 1$). Every (positive) level equivalence is an $\cI$-equivalence, but not vice versa. 

The relevance of $\cI$-spaces comes from the fact that it has an additional monoidal structure not present in spaces, the Day convolution product $X \boxtimes Y$ of
$\cI$-spaces $X$ and $Y$. Using that $\cI$ is symmetric monoidal under
the ordered concatenation of ordered sets
$\bld{m}\concat \bld{n} = \bld{m + n}$, the $\boxtimes$-product is
defined to be the left Kan extension of the $\cI\times \cI$-diagram
$X(-) \times Y(-)$ along the concatenation
$-\concat - \colon \cI \times \cI \to \cI$.  This means that
\[ (X\boxtimes Y)(\bld{n}) \iso \colim_{\bld{k}\concat \bld{l} \to
  \bld{n}} X(\bld{k}) \times Y(\bld{l})\]
where the colimit is taken over the category
$- \concat - \downarrow \bld{n}$. Since $\bld{0}$ is initial in $\cI$, the terminal $\cI$-space $* = \const_{\cI}(*) \iso \cI(\bld{0},-)$ given by the constant $\cI$-diagram on the one point space $*$ is also the unit for $\boxtimes$. Since the twist
isomorphism of the symmetric monoidal structure on $\cI$ is the block
permutation $\chi_{\bld{k},\bld{l}}\colon \bld{k}\concat \bld{l} \to \bld{l}\concat \bld{k}$, also
the twist for $\boxtimes$ takes these permutations into account.

\begin{definition}
  A (commutative) \emph{$\cI$-space monoid} is a (commutative) monoid in
  $(\cS^{\cI},\boxtimes,\ucI)$, and we write $\cC\cS^{\cI}$ for the
  category of commutative $\cI$-space monoids. 
\end{definition}
Unraveling definitions, a commutative $\cI$-space monoid is an $\cI$-space
$M$ together with associative and unital multiplication maps
$M(\bld{k}) \times M(\bld{l}) \to M(\bld{k}\concat \bld{l})$ satisfying
a commutativity condition involving the twist of the $\times$-factors in the source and the map induced by the block permutation  $\bld{k}\concat \bld{l} \to \bld{l}\concat \bld{k}$ in the target.
It is shown in \cite[Theorem 1.2]{Sagave-S_diagram} that up to weak
homotopy equivalence, every $E_{\infty}$ space can be represented by a
commutative $\cI$-space monoid. More precisely, the category
$\cC\cS^{\cI}$ admits a \emph{positive} (projective) $\cI$-model structure with
weak equivalences the $\cI$-equivalences that is related to the category
of $E_{\infty}$ spaces by a chain of Quillen equivalences. The fibrant objects in the positive $\cI$-model structure are called \emph{positive $\cI$-fibrant}. They are characterized by the condition that all maps $\bld{m} \to \bld{n}$ in $\cI$ with $|\bld{m}|\geq 1$ induce weak equivalence between fibrant spaces so that there is no homotopical information in level $0$.

Besides the $\boxtimes$-product, the category of $\cI$-spaces also has a cartesian product $\times$ with unit $*$. 
\begin{definition}
   A \emph{cartesian $\cI$-monoid} is an associative monoid in
   $(\cS^{\cI},\times,\ucI)$ or, equivalently, a functor $\cI
   \rightarrow \cA\cS$ to simplicial or topological monoids.  We will
   denote the category of cartesian $\cI$-monoids by $(\cA\cS)^{\cI}$.
\end{definition}

Every cartesian $\cI$-monoid has an underlying $\cI$-space monoid. To see this, we recall the natural map 
\begin{equation}\label{eq:rhoYX} 
  \rho_{X,Y} \colon X \boxtimes Y  \to X \times Y
\end{equation} from~\cite[Proposition 2.27]{Sagave-S_group-compl}. By definition, it sends a point represented by a tuple $(\alpha \colon \bld{k} \concat \bld{l} \to \bld{n}, x \in X(\bld{k}), y\in Y(\bld{l}))$ to  $((\alpha|_{\bld{k}})_*(x), (\alpha|_{\bld{l}})_*(y)) \in (X \times Y)(\bld{n})$. 
\begin{lemma}\label{lem:boxtimes-times-monoidal}
  The natural map $\rho_{X,Y}$ equips the identity with the
  structure of a lax symmetric monoidal functor $(\cS^{\cI},\boxtimes)
  \to (\cS^{\cI},\times)$.
\end{lemma}
\begin{proof}
  Since $\ucI$ serves as a unit both for $\boxtimes$ and $\times$, the
  identity provides the structure map relating the respective
  units. The compatibility of $\rho_{X,Y}$ with the associativity is
  immediate from the definition.  Compatibility with the symmetry
  isomorphism for the $\boxtimes$-product follows since the latter
  sends the point represented by $\alpha\colon \bld{k}\concat \bld{l}
  \to \bld{n}$, $x \in X(\bld{k})$ and $y\in Y(\bld{l})$ to the point
  represented by $(\alpha \circ \chi_{\bld{k},\bld{l}}, y, x)$.
\end{proof}

\begin{corollary}
Every cartesian $\cI$-monoid $M$ has an underlying $\cI$-space monoid with multiplication 
$M \boxtimes M \xrightarrow{\rho_{M,M}} M \times M \xrightarrow{\mu} M$.\qed
\end{corollary}

The notion of a \emph{commutative} cartesian $\cI$-monoid $M$ is, however, too restrictive for our purposes. It requires all monoids $M(\bld{k})$ to be commutative, which implies for example that their identity components are generalized Eilenberg--Mac\ Lane spaces. In contrast, a commutative $\cI$-space monoid $M$ does not provide monoid structures on the spaces $M(\bld{n})$ unless $\bld{n} = \bld{0}$. Therefore, commutativity with respect to $\boxtimes$ is a less rigid notion.
\begin{definition}\label{i-comm}
  A cartesian $\cI$-monoid $M$ is \emph{$\cI$-commutative} or \emph{Eckmann--Hilton} if its underlying $\cI$-space monoid is commutative. 
\end{definition}

The following lemma follows directly from the definitions:

\begin{lemma}\label{lem:I-commutative}
Let $M$ be a cartesian $\cI$-monoid. The following are equivalent: 
\begin{enumerate}[(i)]
\item $M$ is $\cI$-commutative. 
\item For all injections $\alpha\colon \bld{k} \concat \bld{l} \rightarrow \bld{n}$, the following diagram commutes:
\[\xymatrix@-1pc{ M(\bld{k}) \times M(\bld{l}) \ar[rrr]^{(\alpha|_{\bld{k}})_* \times (\alpha|_{\bld{l}})_*} \ar[d]_{\mathrm{twist}} & && M(\bld{n}) \times M(\bld{n}) \ar[rrd]^{\mu} & & \\
M(\bld{l})  \times M(\bld{k}) \ar[rrr]^{(\alpha|_{\bld{l}})_* \times (\alpha|_{\bld{k}})_*} & & & M(\bld{n}) \times M(\bld{n}) \ar[rr]^{\mu} & & M(\bld{n})}\]
\item For all injections $\alpha\colon \bld{k} \concat \bld{l} \rightarrow \bld{n}$, the map
\[M(\bld{k}) \times M(\bld{l}) \xrightarrow{(\alpha|_{\bld{k}})_* \times (\alpha|_{\bld{l}})_*} M(\bld{n}) \times M(\bld{n}) \xrightarrow{\mu}  M(\bld{n})\]
is a monoid homomorphism.\qed
\end{enumerate}
\end{lemma}

\begin{remark}
In~\cite[Definition 3.1]{DP-I}, Dadarlat and Pennig introduce \emph{Eckmann--Hilton $\cI$-groups} which are a special case of our $\cI$-commutative cartesian $\cI$-monoids.
They view the $\cI$-space monoid structure as additional data (rather than deriving it from the cartesian $\cI$-monoid structure) and require the $M(\bld{k})$ to be groups (rather than monoids). Schwede~\cite[Definition 2.3.4]{schwede-global} studies \emph{symmetric monoid-valued orthogonal spaces} which are the orthogonal counterparts of $\cI$-commutative cartesian $\cI$-monoids and therefore provide examples of the latter.  
\end{remark}
We write $B^{\boxtimes}(M) = |[q] \mapsto M^{\boxtimes q} |$ for the bar construction of an $\cI$-space monoid with respect to $\boxtimes$ and
$B^{\times}(M) = |[q]\mapsto M^{\times q} |$ for the bar construction of a cartesian $\cI$-monoid with respect to $\times$.  The discussion before~\cite[Proposition 4.2]{Sagave-S_group-compl} provides a comparison map
\begin{equation}\label{eq:c_M}
c_M \colon B^{\boxtimes}(M)  \to B^{\times}(M)
\end{equation}
which is the realization of the map of simplicial objects given in each 
simplicial degree by the map 
$M\boxtimes \dots \boxtimes M \to M \times \dots  \times M$ induced by~\eqref{eq:rhoYX}.

\begin{lemma}\label{lem:commutative-bar-constr-comp-for-M}
  Let $M$ be an $\cI$-commutative cartesian $\cI$-monoid. Then the 
  bar constructions $B^{\times}(M)$ and $B^{\boxtimes}(M)$ inherit the structure
  of a commutative $\cI$-space monoid, and
  $c_M \colon B^{\boxtimes}(M) \to B^{\times}(M)$ is a map of commutative
  $\cI$-space monoids. 
\end{lemma}
\begin{proof}
  Since $M$ is $\cI$-commutative, the $q$-fold cartesian product
  $M^{\times q}$ with its canonical monoid structure is also
  $\cI$-commutative by Lemma~\ref{lem:I-commutative} (ii). Thus we can
  view $M^{\times q}$ as a commutative $\cI$-space monoid. Next we
  note that the simplicial structure maps of the simplicial object
  $[q]\mapsto M^{\times q}$ are morphism of $\cI$-space monoids. In
  the case of the injection $\delta_1 \colon [1] \to [2]$, this
  amounts to the commutativity of the square
  \[ \xymatrix@-1pc{ (M \times M) \boxtimes (M\times M) \ar[d]\ar[rr] & & M \times M \ar[d] \\
      M \times M \ar[rr] & & M }\] which can be checked using
  Lemma~\ref{lem:I-commutative}. The case of other simplicial
  structure maps is analogous. Since the realization of a simplicial
  object in commutative $\cI$-space monoids is a commutative
  $\cI$-space monoid (see~\cite[Proposition 9.9]{Sagave-S_diagram}), it follows that
  $B^{\times}(M)$ inherits the structure of a commutative $\cI$-space
  monoid. The commutative $\cI$-space monoid structure on $B^{\boxtimes}(M)$ is
  discussed in~\cite[\S{}4.1]{Sagave-S_group-compl}.

  It remains to check that $c_M$ is a morphism of commutative
  $\cI$-space monoids. For this, one can use
  Lemma~\ref{lem:I-commutative} to see that the iteration
  $M^{\boxtimes q} \to M^{\times q}$ of the map~\eqref{eq:rhoYX} is a
  morphism of commutative $\cI$-space monoids. Since it is compatible
  with the simplicial structure maps, the statement about $c_M$ follows.
\end{proof}

\subsection{Grouplike cartesian \texorpdfstring{$\cI$}{I}-monoids}
Via the passage to the underlying infinite loop space, connective
spectra are equivalent to grouplike $E_{\infty}$ spaces. In view of
the equivalence between commutative $\cI$-space monoids and
$E_{\infty}$ spaces, an analogous statement holds for grouplike commutative
$\cI$-space monoids (see~\cite[Theorem~1.5]{Sagave-S_group-compl}). Here an $\cI$-space monoid $M$ is \emph{grouplike} if the monoid $\pi_0(M_{h\cI})$ is a group. 

We say that a cartesian $\cI$-monoid $M$ is \emph{grouplike} if its
underlying $\cI$-space monoid is grouplike. It will be convenient
to have an intrinsic criterion for when a cartesian $\cI$-monoid is
grouplike. To formulate it, we write $\cN$ for the subcategory of
$\cI$ with morphisms all subset inclusions. Recall that an $\cI$-space
$X$ is \emph{semistable} if a (and hence any) fibrant replacement in the
absolute $\cI$-model structure induces a weak equivalence on the
homotopy colimits of the underlying $\cN$-diagrams~\cite[\S
2.5]{Sagave-S_group-compl}. Any fibrant $\cI$-space is semistable.
Another simple sufficient condition for semistability is convergence, that is, the existence of a non-decreasing sequence $(\lambda_k)_{k\geq 0}$ with $\lim_{k\to\infty}\lambda_k = \infty$ such that any morphism $\bld{m}\to\bld{n}$ in $\cI$ with $|\bld{m}|\geq k$ induces a $\lambda_k$-connected map $X(\bld{m}) \to X(\bld{n})$.

\begin{proposition}
  Let $M$ be a cartesian $\cI$-monoid such that $M$ is semistable as
  an $\cI$-space and the monoid
  $\colim_{\bld{k}\in \cN}\pi_0M(\bld{k})$ is a group. Then $M$ is grouplike.
\end{proposition}
\begin{proof}
It suffices to prove the statement working over simplicial sets. We equip the  category of simplicial monoids $\cA\cS$ with the standard model structure where a map is a fibration or weak equivalence if the underlying map in $\cS$ is. 
Using \cite[Theorem 5.2]{Dugger_replacing}, this model structure induces a hocolim-model structure on the category $(\cA\cS)^{\cI}$ of $\cI$-diagrams in $\cA\cS$ where the weak equivalences are detected by the corrected homotopy colimits over $\cI$ in $\cA\cS$. We say that a cartesian $\cI$-monoid $A$ is \emph{semistable} if a fibrant replacement in this model structure is an $\cN$-equivalence, i.e., if it induces a weak equivalence when forming the corrected homotopy colimit in $\cA\cS$ of the underlying $\cN$-diagrams. The corrected homotopy colimit of an $\cN$-diagram in $\cA\cS$ (resp. $\cS$) is weakly equivalent to the colimit of a cofibrant replacement in the corresponding projective level model structure on $\cN$-diagrams in $\cA\cS$ (resp. $\cS$). Since the forgetful functor $\cA\cS \to \cS$ preserves cofibrations and sequential colimits, this implies that a map in $(\cA\cS)^{\cI}$ is an $\cN$-equivalence in $(\cA\cS)^{\cI}$ if and only if the underlying map in $\cS^{\cI}$ is an $\cN$-equivalence in $\cS^{\cI}$. The implication (iii) $\Rightarrow$ (ii) in~\cite[Proposition 2.10]{Sagave-S_group-compl} also holds for $\cI$-diagrams of simplicial monoids since we can replace $\cS$ by $\cA\cS$ everywhere in the argument given there. It follows that $A$ is semistable in $(\cA\cS)^{\cI}$ if it is semistable in $\cS^{\cI}$. 

To prove the statement of the proposition, we now form a zig-zag 
\[ M \to M^f \ot F_{\bld{0}}^{\cI}( (M^f(\bld{0}))^{\mathrm{cof}}) \]
where the first map is a fibrant replacement in $(\cA\cS)^{\cI}$ and the second map is the 
derived adjunction counit of the Quillen equivalence $ F_{\bld{0}}^{\cI} \colon \cA\cS \rightleftarrows (\cA\cS)^{\cI} \colon \mathrm{Ev}_0$. Then the first map is an $\cN$-equivalence of cartesian $\cI$-monoids by the above discussion, and the second map is an absolute level equivalence since both objects are fibrant in $(\cA\cS)^{\cI}$. Since the condition about $\pi_0$ is preserved under $\cN$-equivalences, this reduces the claim to the case of a constant $\cI$-diagram of simplicial monoids where it is easily verified. 
\end{proof}
\begin{remark}
The semistability assumption in the proposition is necessary: For the cartesian $\cI$-monoid $\Sigma$ formed by the various symmetric groups one finds $\pi_0\Sigma_{h\cI}$ to be given by the conjugacy classes in $\Sigma_{(\infty)} = \colim_{k \in \cN}\Sigma_k$, with the monoid structure induced by the block sum of permutations. Clearly this monoid is not a group. In fact, the cycle decomposition of permutations shows  $\pi_0\Sigma_{h\cI} \cong \mathbb N^{(\infty)}$. In particular, the projection $\pi_0M_{h\mathcal N} \rightarrow \pi_0M_{h\cI}$ is not generally a homomorphism.
\end{remark}

\section{Parametrized spectra and  \texorpdfstring{$\cI$}{I}-monoid actions}\label{sec:parametrized}
We briefly recall the model for parametrized spectra introduced in~\cite{HSS-retractive}. To keep the exposition simple, we focus on the point set level definitions and constructions in this section and only discuss homotopical properties in the next section. 

\subsection{Symmetric spectra in retractive spaces} Let $\cS_{\cR}$ be
the category of retractive spaces. Its objects are pairs of spaces $(U,K)$
with structure maps $K \to U \to K$ that compose to the identity. Its 
morphisms are the pairs of maps that make the obvious two squares
commutative. We often view a based space $S$ as a retractive space and simply
write $S$ for $(S,*)$.

The category $\cS_{\cR}$ is symmetric monoidal under the fiberwise
smash product $\barsm$ with monoidal unit the retractive space
$S^0 = (S^0, *)$ (see~\cite[Definition~2.21]{HSS-retractive} for
details). On base spaces, it is just the cartesian product and restricted
to retractive spaces coming from based spaces it is the usual smash product.

Given an object $T$ in a closed symmetric monoidal category
$(\cC,\tensor)$, one can form symmetric spectra in $\cC$ with
$- \tensor T$ as suspension. They are defined as sequences
$(X_n)_{n\geq 0}$ of objects in $\cC$ with an action of the symmetric
group $\Sigma_n$ on $X_n$ and suitably compatible structure maps
$X_n \tensor T \to X_{n+1}$ (see~\cite[Definition~7.2]{Hovey_symmetric-general}).
\begin{definition}
  The category of \emph{symmetric spectra in retractive spaces}
  $\SpsymR$ is the category of symmetric spectrum objects in
  $(\cS_{\cR},\barsm, S^0)$ with $-\barsm (S^1,*)$ as suspension.
\end{definition}
Symmetric spectrum objects in unbased spaces with $- \times * \iso \id$ as suspension functor are equivalent to $\cI$-spaces. This implies that the projection to the base space $\pi_b \colon \cS_{\cR} \to \cS, (U,K) \to K$ induces 
a functor $\pi_b\colon \SpsymR \to \cS^{\cI}$. 
\begin{definition}
Let $X$ be an $\cI$-space. The category of $X$-relative symmetric spectra $\Spsym{X}$ is the fiber of $\pi_b\colon \SpsymR \to \cS^{\cI}$ over $X$. 
\end{definition}
We stress that the objects in $\Spsym{X}$ are in general not the spectrum objects in some category, but rather sequences of retractive spaces over varying base spaces. When $X$ is the terminal $\cI$-space, $X$-relative symmetric spectra are equivalent to ordinary symmetric spectra which we can therefore view as a subcategory of $\SpsymR$. A map of $\cI$-spaces $f \colon X \to Y$ induces an adjunction $f_! \colon \Spsym{X} \rightleftarrows \Spsym{Y} \colon f^*$ that is in level $n$ given by pushforward and pullback along $f(\bld{n})\colon X(\bld{n}) \to Y(\bld{n})$. 

The category $\SpsymR$ furthermore comes with symmetric monoidal Day convolution
smash product $\barsm$ induced by the fiberwise smash product $\barsm$
of $\cS_{\cR}$. Explicitly, it is given by 
\begin{equation}\label{eq:barsm-def} ((E,X)\barsm (E',X'))_n = \colim_{\alpha\colon \bld{k} \concat
    \bld{l} \to \bld{n}} (E,X)_k \barsm (E',X')_l \barsm
  S^{\bld{n}-\alpha}
\end{equation}
where the colimit is taken over the comma category
$-\concat - \downarrow \bld{n}$ and $S^{\bld{n}-\alpha}$ denotes a
smash power of $S^1$ indexed by $\bld{n}-\alpha$, the complement of
the image of $\alpha$. (We refer to~\cite[\S 4]{HSS-retractive} for the description
of the maps in the colimit system.) The unit of the $\barsm$-product on
$\SpsymR$ is $\bS = (\bS,*)$. On base $\cI$-spaces, the
$\barsm$-product is just the $\boxtimes$-product of $\cI$-spaces. When
$M$ is a commutative $\cI$-space monoid, the $\barsm$-product and the
pushforward along the multiplication $\mu \colon M \boxtimes M \to M$
induce the symmetric monoidal \emph{convolution smash product}
\[ \Spsym{M}\times\Spsym{M}\xrightarrow{\barsm}\Spsym{M\boxtimes{}M}\xrightarrow{\mu_!}\Spsym{M}
\]
on $\Spsym{M}$. The possibility to implement this monoidal product is one of the major advantages of working relative to base $\cI$-spaces rather than relative to base spaces.

The levelwise collapse of base spaces induces a functor $\Theta\colon \SpsymR \to \Spsym{}$. In the above notation, we have $\Theta(E,X) = (X \to *)_!(E,X)$. This functor $\Theta$ is strong symmetric monoidal.

\subsection{Actions of cartesian \texorpdfstring{$\cI$}{I}-monoids} 
We will now study actions of cartesian $\cI$-monoids on symmetric
spectra. To this end, we recall from~\cite[Section~4.17]{HSS-retractive} that there is a functor
\begin{equation}\label{eq:tensor-SI-SpsymR} 
-\times-\colon \cS^{\cI}\times\SpsymR\to\SpsymR 
\end{equation}
sending an $\cI$-space $Y$ and a symmetric spectrum in retractive space $(E,X)$ to the object $Y \times (E,X)$ given in level $n$ by $(Y(\bld{n})\times{}E_n,Y(\bld{n})\times{}X(\bld{n}))$. It exhibits $\SpsymR$ as being tensored over $(\cS^{\cI},\times, *)$. 

One important source of parametrized spectra is the functor 
\[ \bS^{\cI}_t \colon \cS^{\cI} \to \SpsymR, \quad \bS^{\cI}_t[X] = X \times \bS. \] It is strong symmetric monoidal with respect to $\boxtimes$ and $\barsm$, that is, there is a natural isomorphism $\bS^{\cI}_t[X\boxtimes Y] \iso \bS^{\cI}_t[X] \barsm \bS^{\cI}_t[Y]$. There is also a natural isomorphism $\Theta(\bS^{\cI}_t[X]) \iso \bS^{\cI}[X]$ identifying the symmetric spectrum obtained by collapsing the base $\cI$-space of $\bS^{\cI}_t[X]$ with the suspension spectrum $\bS^{\cI}[X]$ of the $\cI$-space $X$ considered in \cite[\S{}3.17]{Sagave-S_diagram}.

For the next definition we again view an ordinary symmetric spectrum $E$ as an object in $\SpsymR$ with the terminal $\cI$-space $*$ as base.  
\begin{definition}\label{def:cartesian-acts-on-spsym}
Let $M$ be a cartesian $\cI$-monoid and let $E$ be a symmetric spectrum. An $M$-action on $E$ is an  associative and unital map $\mu\colon M \times E \to E$ in~$\SpsymR$. 
\end{definition}
Unraveling definitions, $\mu\colon M \times E \to E$ consists of actions $M(\bld{n}) \times E_n \to E_n$ for all $\bld{n}$ in $\cI$ that are associative, unital, base point preserving, and compatible with the structure maps of $M$ and $E$. For a constant cartesian $\cI$-monoid, this is just the usual notion of an
action of a simplicial or topological monoid on $E$.

To our knowledge, actions of cartesian $\cI$-monoids on symmetric spectra were first considered in \cite{DP-I}. We will discuss their approach below. A similar, but more restrictive notion appears briefly \cite[Chapter 23]{May-S_parametrized} and is used there to produce commutative Thom spectra, which we (re)obtain by application of $\Theta$ to the following construction.

\begin{definition} Let $M$ be a cartesian $\cI$-monoid $M$ that acts on a symmetric spectrum $E$. The \emph{homotopy quotient} of this action is defined to be the two sided bar construction
\begin{equation}\label{eq:hty-quotient}
E \sslash M = B^{\times}(\ucI,M,E) = \big|[q] \mapsto \ucI\times M^{\times q} \times E \big|. 
\end{equation}
As usual, the simplicial structure maps of the underlying simplicial object are induced by the unit of $M$, the multiplication of $M$, and the action of $M$ on $E$. 
\end{definition}
We note that the underlying $\cI$-space of $E\sslash M$ is $B^{\times}(M)$, the bar construction of $M$ with respect to $\times$. As we will see in Sections~\ref{sec:examples-of-actions} and~\ref{finalsec}, the homotopy quotient is the source of many interesting examples for parametrized spectra. 

\begin{remark} For constant $M$ this construction appears in \cite[Chapter 22.1]{May-S_parametrized} and the object $\mathbb S \sslash \GLoneIof{\mathbb S}$ makes a brief appearance in \cite[Definition 23.5.1]{May-S_parametrized}. In their notation, the latter is the PFSP $B^{\times}(*,F,S)$. Pullbacks thereof seem to be the only overlap between the objects considered in loc.cit. and our $E \sslash M$, however.
\end{remark}
We now examine the formal structure of homotopy quotients and relate them to bar constructions involving the fiberwise smash product. This analysis will form the basis for the proof of Theorem \ref{thm:RsslashH-identification}. 

By passing to the underlying $\cI$-space monoid and applying the functor 
$\bS^{\cI}_t$, a cartesian $\cI$-monoid $M$ also gives rise to an associative
parametrized ring spectrum $\bS^{\cI}_t[M]$ with multiplication
\[\bS^{\cI}_t[M]\barsm \bS^{\cI}_t[M] \xrightarrow{\iso}
\bS^{\cI}_t[M\boxtimes M] \to \bS^{\cI}_t[M\times M] \to
\bS^{\cI}_t[M].\]
Furthermore, there is a natural distributivity map 
\begin{equation}\label{eq:distr}
\delta \colon (Y \times (E,X)) \barsm (Y' \times (E',X')) \to  (Y \boxtimes Y') \times \left((E,X) \barsm (E',X')\right)
\end{equation}
for $\cI$-spaces $Y,Y'$ and parametrized symmetric spectra $(E,X),(E',X')$, introduced in~\cite[\S 4.17]{HSS-retractive}. In the description of the $\barsm$-product given in~\eqref{eq:barsm-def}, it is given on the term of the colimit system indexed by $\alpha\colon \bld{k} \concat \bld{l} \to \bld{n}$ by the composite
\[
 (Y(\bld{k}) \times (E,X)_k) \barsm (Y'(\bld{l}) \times
(E',X')_l) \barsm S^{\bld{n}-\alpha} \to (Y \boxtimes
Y')(\bld{n}) \times ( (E,X) \barsm (E',X'))_n
\]
of the twist of the middle factors and the maps into the $\boxtimes$ and $\barsm$-products determined by $\alpha$. 
By setting  $(E,X) = \bS$, and $Y' = *$, the distributivity map specializes to a natural transformation of functors $\cS^{\cI}\times\SpsymR\to\SpsymR$ of the form
\begin{equation} 
\label{eq:rhoZYEX} \rho_{Y, (E',X')} \colon \bS^{\cI}_{t}[Y]\barsm{}(E',X') 
\to Y\times{}(E',X')
\end{equation} 
relating the two actions. The following compatibility condition is also stated in~\cite[(4.23)]{HSS-retractive} and will be used several times:
\begin{lemma}\label{lem:distr-square}
The natural maps~\eqref{eq:distr} and~\eqref{eq:rhoZYEX} participate in a commutative square whose top horizontal map is the twist of the middle factors:
  \[\xymatrix@-1pc{
\bS^{\cI}_t[Y] \barsm (E,X)  \barsm \bS^{\cI}_t[Y'] \barsm (E',X') \ar[rrr]^{\iso} \ar[d]^{\rho_{Y,(E,X)} \barsm \rho_{Y',(E',X')}}&&&  
\bS^{\cI}_t[Y \boxtimes Y'] \barsm (E,X) \barsm (E',X') \ar[d]_{\rho_{Y\boxtimes Y',(E,X)\barsm (E',X')}} \\
(Y \times (E,X) ) \barsm (Y' \times (E',X') ) \ar[rrr]^{\delta} &&&  (Y \boxtimes Y') \times  ((E,X) \barsm (E',X'))}
\]
\end{lemma}
\begin{proof}
  The iterated $\barsm$-product in the source can be written as a colimit of terms of the form
  \[(Y(\bld{k}) \times (S^k,*)) \barsm (E,X)_l \barsm (Y(\bld{k'}) \times (S^{k'},*)) \barsm (E',X')_{l'} \barsm S^{\bld{n}-\alpha}  \]
  indexed over injections $\alpha \colon \bld{k} \concat\bld{l}\concat \bld{k'} \concat \bld{l'} \to \bld{n}$. It can be checked from the definitions that both composites map these terms into $(Y \boxtimes Y)(\bld{n}) \times ((E,X) \barsm (E',X'))_{n}$  by the canonical map induced by $\alpha$. 
\end{proof}

In total, if a cartesian $\cI$-monoid $M$ acts on a symmetric spectrum $E$, we  get a map 
\[ \bS^{\cI}_t[M] \barsm E \xrightarrow{\rho_{M,E}} M \times E \xrightarrow{\mu} E.\]
\begin{lemma}\label{lem:module-structure-from-action}
This map provides an $\bS^{\cI}_t[M]$-module structure on $E$. 
\end{lemma}
\begin{proof}
  Unitality is clear, and associativity follows by using  the associativity of the action of $M$ on $E$ and the commutative square in Lemma~\ref{lem:distr-square} with $(E,X) = \bS$, $(E',X')=E$, and $Y = Y' =  M$.
\end{proof}

We now relate the $M$-action and the $\bS^{\cI}_t[M]$-module structure
on $E$. The module structure of the previous lemma and the
$\bS^{\cI}_t[M]$-module structure on $\bS$ induced by $M \to *$ allow
us to form the two sided bar construction $B^{\barsm}(\bS,\bS^{\cI}_t[M],E)$
with respect to $\barsm$. Moreover, the natural
transformation~\eqref{eq:rhoZYEX} induces maps
$\bS \barsm \bS^{\cI}_t[M]^{\barsm q} \barsm E \to \ucI \times
M^{\times q} \times E$ that are compatible with the simplicial
structure maps by the commutativity of the square in
Lemma~\ref{lem:distr-square}. Passing to geometric realizations, we
obtain a natural map
\begin{equation}\label{eq:bar-constr-comparison}
B^{\barsm}(\bS,\bS^{\cI}_t[M],E) \to  B^{\times}(\ucI,M,E) = E \sslash M
\end{equation}
The underlying map of $\cI$-spaces is the map $c_M \colon  B^{\boxtimes}(M) \to B^{\times}(M)$ from~\eqref{eq:c_M}. We will show in Proposition~\ref{prop:bar-comparison} that a suitably derived version of~\eqref{eq:bar-constr-comparison} is a local equivalence.

\subsection{The Eckmann-Hilton condition} 
The next definition is analogous to~\cite[Definition 3.7]{DP-I}.
\begin{definition}\label{def:EH-for-R}
Let $M$ be a cartesian $\cI$-monoid acting on both a symmetric ring spectrum $R$ and an $R$-module $E$. We say that the actions have the \emph{Eckmann--Hilton} property if the square 
\begin{equation}\label{eq:MR-Eckmann--Hilton}\xymatrix@-1pc{ (M \times R) \barsm (M \times E) \ar[rr] \ar[d]&& (M\boxtimes M) \times (R \sm E) \ar[rr] && M \times E \ar[d] \\ 
R \sm E \ar[rrrr] &&&& E }
\end{equation}
is commutative. Here the maps are given by the action of $M$, the multiplication of $M$ and the action of $R$ on $E$, and an instance of the distributivity map~\eqref{eq:distr}. 
\end{definition}
If $R = E$ and the two actions agree, we shall simply speak of an Eckmann-Hilton action of $M$ on $R$. The following lemma is proved analogously to Lemma~\ref{lem:commutative-bar-constr-comp-for-M}. 
\begin{lemma}\label{lem:BSIMR-to-RsslashM-commutative}
Let $M$ be an $\cI$-commutative cartesian $\cI$-monoid acting both on a (commutative) symmetric ring spectrum $R$ and an $R$-module $E$, and the suppose that the actions have the Eckmann--Hilton property. Then
\begin{enumerate}[(i)]
\item $R \sslash M$ inherits a (commutative) ring structure, 
\item the comparison map $B^{\barsm}(\bS, \bS^{\cI}_t[M],R) \to  R\sslash M$ is a map of (commutative) parametrized ring spectra, 
\item the homotopy quotient $E \sslash M$ inherits an $R\sslash M$-module structure, and
\item the map $B^{\barsm}(\bS,\bS^{\cI}_t[M],E) \to E\sslash M$ is an $B^{\barsm}(\bS,\bS^{\cI}_t[M],R)$-module map. \qed
\end{enumerate}
\end{lemma}

\subsection{Linear actions of cartesian \texorpdfstring{$\cI$}{I}-monoids}
When $R$ is a symmetric ring spectrum, $E$ is a left $R$-module spectrum, and $Y$ is an $\cI$-space, then we may view $Y \times E$ as an $R$-module. To see this, we note that the $R$-module structure is given by suitably compatible maps $\mu_{\alpha}\colon R_{k} \sm E_{l} \sm S^{\bld{n}-\alpha}\to E_{n}$ indexed by morphisms $\alpha \colon \bld{k}\concat \bld{l} \to \bld{n}$ in $\cI$. The $R$-module structure on $Y \times E$ is then given by the maps 
\[ 
 R_{k} \barsm  (Y(\bld{l}) \times E_{l}) \barsm S^{\bld{n}-\alpha}\xrightarrow{\iso}  Y(\bld{l}) \times (R_{k} \barsm E_{l} \barsm S^{\bld{n}-\alpha})\xrightarrow{(\alpha|_{\bld{l}})_*\barsm \mu_{\alpha} } Y \times E_{n}
\]
where we view all pointed spaces as retractive spaces and $Y(\bld{l}) \times (-)$ denotes the product with $ Y(\bld{l})$ both on the base and total space. 
\begin{definition}
The action $\mu$ of a cartesian $\cI$-monoid $M$ on an $R$-module $E$ is \emph{$R$-linear} if $\mu\colon M \times E \to E $ is an $R$-module map.  
\end{definition}
Now suppose $M$ acts $R$-linearly on $R$ itself (which is for example the case if the action is Eckmann--Hilton in the sense of Definition~\ref{def:EH-for-R}).
By composing with a suitable instance of \eqref{eq:rhoZYEX}, the $R$-module map $\mu \colon M \times R \to R$ induces an $R$-module map $\bS^{\cI}_t[M] \barsm R \to M \times R \to R $. 
Via the adjunction given by restriction and extension along $\bS \to R$, the latter map has the following adjoint in $\SpsymR$:  
\begin{equation}\label{eq:mu-flat} \mu^{\flat} \colon \bS^{\cI}_t[M] \to R
\end{equation}
\begin{lemma}
If $\mu$ is an $R$-linear action, then $\mu^{\flat}$ is a map of associative para\-metrized ring spectra $\mu^{\flat} \colon \bS^{\cI}_t[M] \to R^\op$. 
\end{lemma}
\begin{proof}
It is clear that $\mu^{\flat}$ is unital. The $R$-linearity of the action implies that the square
\[ \xymatrix@-1pc{R \barsm \bS^{\cI}_t[M] \barsm R \ar[rr]^-{\mathrm{id} \barsm \mu} \ar[d]_{\tau \barsm \mathrm{id}}&& R \barsm R \ar[d]\\
                  \bS^{\cI}_t[M] \barsm R \barsm R \ar[r]& \bS^{\cI}_t[M] \barsm R \ar[r]^-\mu& R}\]
commutes, where the unlabeled maps denote the multiplication of $R$. Precomposing with $\mu^\flat$ and the unit $\bS \rightarrow R$ on the left and right $R$-factor, respectively, one can use the $R$-linearity of $\mu$ in the upper composition and the associativity of $\mu$ on the lower one to obtain the claim.
\end{proof}

Since $R$ has the terminal $\cI$-space as base, parametrized (ring) spectrum maps $\bS^{\cI}_t[M] \to R^\op$ are in one to one correspondence to (ring) spectrum maps $\bS^{\cI}[M] \to R^\op$. Using the adjunction $(\bS^{\cI},\Omega^{\cI})$ from \cite[\S{}3.17]{Sagave-S_diagram}, the action map $\mu$ thus corresponds to a map of $\cI$-space monoids \begin{equation}\label{eq:mu-sharp} \mu^{\sharp} \colon M \to \Omega^{\cI}(R^\op).
\end{equation}
\begin{lemma}\label{dp-map}
  In $\cI$-space level $\bld{k}$, $\mu^{\sharp}$ sends $m \in M(\bld{k})$ to the map
\[S^{k} \to R_k, \quad s \mapsto \mu_{\bld{k}}(m, \iota_{k}(s)) \]
where $\iota_{k}$ denotes the $k$th component of the unit map $\bS \to R$. \qed
\end{lemma}

\begin{remark}
The observation that (Eckmann-Hilton-) actions of cartesian $\cI$-monoids on symmetric spectra allow to extract this map $\mu^\sharp$ was first made by  Dadarlat and Pennig; see ~\cite[Theorem 3.8]{DP-I}. 
\end{remark}

\begin{proposition}\label{prop:BSSIME-to-E-sslash-M}
Let $M$ be an cartesian $\cI$-monoid acting both on a commutative symmetric ring spectrum $R$ and an $R$-module $E$, and suppose that both actions have the Eckmann--Hilton property. Then the $\bS^{\cI}_t[M]$-module structure induced by the $M$-action on $E$ coincides with the restriction of the $R$-module structure along $\bS^{\cI}_t[M] \to R$.
\end{proposition}
\begin{proof}
The claim is equivalent to requiring that the square of $R$-module spectra 
\[\xymatrix@-1pc{ \bS^{\cI}_t[M] \barsm R \barsm E \ar[r] \ar[d] & \bS^{\cI}_t[M] \barsm E \ar[d] \\ 
R \sm E \ar[r] & E} \]
commutes (where the left $R$-action in the upper left corner is through the middle term). This follows
from the commutativity of the square in Lemma~\ref{lem:distr-square} (where one sets ${Y = M}$, ${Y' = \ucI}$, $(E,X) = R$, and $(E',X') = E$) and the restriction of the Eckmann--Hilton condition for $E$ along $M \times \ucI \to M\times M$. 
\end{proof}

\section{Homotopical comparison of products and bar constructions}\label{sec:product-compar}
In order to analyze the homotopical properties of the constructions from the last section and to prove Theorem \ref{thm:RsslashH-identification} in Section \ref{proofs}, we briefly review some key aspects of the homotopy theory of symmetric spectra in retractive spaces. The formulations of almost all results of the present paper only involve the different types of weak equivalences for symmetric spectra in retractive spaces. We therefore focus our exposition on these notions and refer to \cite{HSS-retractive} for details about the model structures in which they participate. 

\subsection{Homotopy theory of symmetric spectra in retractive spaces} We say that a map $(U,K) \to (V,L)$ of retractive spaces is a \emph{weak equivalence} if both of its components $U \to V$ and $K \to L$ are weak homotopy equivalences of spaces. This notion leads to various notions of weak equivalences for symmetric spectra in retractive spaces. We say that a map $(E,X) \to (E',X')$ in $\SpsymR$ is an absolute (resp. a positive) level equivalence if $(E,X)_n \to (E',X')_n$ is a weak equivalence of retractive spaces for all $n \geq 0$ (resp. all $n \geq 1$). It follows from the definitions that if $(E,X) \to (E',X')$ is an absolute or positive level equivalence in $\SpsymR$, then the underlying map of $\cI$-spaces has this property.

To get to the most relevant notion of weak equivalence, one can apply Hovey's general construction principle for model structures on categories of generalized symmetric spectra \cite[Theorem 8.2]{Hovey_symmetric-general} to $\SpsymR$. It provides an \emph{absolute local model structure} on $\SpsymR$. (We refrain from calling it \emph{stable} since in lack of a zero object, $\SpsymR$ is not stable is the sense that suspension is invertible.) The weak equivalences in this model structure are called~\emph{local equivalences}. Both absolute and positive level equivalences are local equivalences.
Since a functor that is both left and right Quillen preserves all weak equivalences, it follows from~\cite[Corollary~5.13]{HSS-retractive} that if $ (E,X) \to (E',X')$ is a local equivalence, then the underlying map of base $\cI$-spaces is an $\cI$-equivalence. We also note that the collapse of base space functor $\Theta\colon \SpsymR \to \Spsym{}$ sends level and hence local equivalences $(E,X) \to (F,Y)$ to stable equivalences if we require the sections of the retractive spaces $(E,X)_n$ and $(F,Y)_n$ to be $h$-cofibrations in the topological case. 
% One can also characterize the cofibrations (see~\cite[Proposition 8.5]{Hovey_symmetric-general}) and the fibrant objects (see~\cite[Section 5.7]{HSS-retractive}) in this model structure, but their precise form is not relevant for the applications in the present paper.  

One main result about the homotopy theory of parametrized symmetric spectra is~\cite[Theorem~1.2]{HSS-retractive} which states that the local model structure on $\SpsymR$ restricts to \emph{absolute local} model structures on the subcategories $\Spsym{X}$. That is, there is a local model structure on $\Spsym{X}$ where a map is a weak equivalence, cofibration or fibration if and only if it is so as a map in $\SpsymR$. The adjunction $(f_!,f^*)$ induced by a map of $\cI$-spaces $f\colon X \to Y$ is a Quillen adjunction with respect to these model structures, and a Quillen equivalence if $f$ is an $\cI$-equivalence. The chain of $\cI$-equivalences relating an $\cI$-space $X$ to the constant $\cI$-diagram on its Bousfield--Kan homotopy colimit $X_{h\cI}$ induces a chain of Quillen equivalences between $\Spsym{X}$ and the stabilization of the category of spaces under and over $X_{h\cI}$~\cite[Corollary~5.24]{HSS-retractive}.

\subsection{Homotopical comparison of products} To compare the different products in the last section, we first work over simplicial sets in order to rely on the results from \cite[Section~8]{HSS-retractive}. We explain in Corollary~\ref{cor:top-comparison-cor} how to derive a topological version of our results. 

Consider the map $\rho_{Y,(E,X)} \colon \bS^{\cI}_{t}[Y]\barsm{}(E',X') 
\to Y\times{}(E',X')$ from~\eqref{eq:rhoZYEX}. In general, it cannot be an isomorphism because the underlying map of $\cI$-spaces fails to be so. However, there is a  useful criterion for when it gives rise to a local equivalence. To phrase it, we use the notion of \emph{flat} $\cI$-spaces. This is a mild cofibrancy condition that is satisfied both by cofibrant $\cI$-spaces and the underlying $\cI$-spaces of cofibrant commutative $\cI$-space monoids~\cite[\S 3.8]{Sagave-S_diagram}. 

\begin{proposition} \label{prop:homotopical-comparison-product}
Let $(E,X)$ be a symmetric spectrum in retractive spaces, let $Y$ be a positive $\cI$-fibrant $\cI$-space, and let $h\colon Y^{\mathrm{cof}} \to Y$ be a positive level equivalence with $Y^{\mathrm{cof}}$ flat as an $\cI$-space. Then the composite of 
\[ \bS^{\cI}_t[Y^{\mathrm{cof}}] \barsm (E,X) \xrightarrow{\bS^{\cI}_t[h] \barsm \id}  \bS^{\cI}_t[Y] \barsm (E,X) \xrightarrow{\rho_{Y,(E,X)}} Y \times (E,X) \]
is a local equivalence. 
\end{proposition}
\begin{proof}
For any $\cI$-space $Z$ that is homotopy constant in positive degrees, a cofibrant replacement and the Quillen equivalence \cite[Theorem 3.3]{Sagave-S_diagram} induce a zig-zag of positive level equivalences $Z \ot Z^{\mathrm{cof}}\to \const_{\cI} \colim_{\cI}Z^{\mathrm{cof}}$ relating $Z$ to a constant $\cI$-space. The functor $- \times (E,X)$ sends positive level equivalences of $\cI$-spaces to positive level equivalences in $\SpsymR$ and thus in particular to local equivalences. Moreover, $ \bS^{\cI}_t[-] \barsm (E,X)$ sends $\cI$-equivalences between flat $\cI$-spaces to local equivalences by \cite[Corollary~8.9]{HSS-retractive}. Together this reduces the claim to showing that $ \bS^{\cI}_t[Y] \barsm (E,X) \to Y \times (E,X)$ is a local equivalence when $Y$ is the constant $\cI$-space on a space. In this case, it is even an isomorphism by inspection of the colimit defining $\barsm$.
\end{proof}

\begin{proposition}\label{prop:bar-comparison}
  Let $E$ be a symmetric spectrum with the action of a positive
  $\cI$-fibrant cartesian $\cI$-monoid $M$, and let $M^{\mathrm{cof}} \to M$ be a
  positive level equivalence of $\cI$-space monoids with $M^{\mathrm{cof}}$ flat as
  an $\cI$-space. Then the composite of
  \[
  B^{\barsm}(\bS,\bS^{\cI}_t[M^{\mathrm{cof}}],E) \to B^{\barsm}(\bS,\bS^{\cI}_t[M],E) \to
  B^\times(*,M,E)
  \]
  is a local equivalence. In particular, $c_M$ induces an $\cI$-equivalence $B^{\boxtimes}(M^{\mathrm{cof}}) \to B^{\times}\!(M)$.
\end{proposition}
\begin{proof}
  This follows by applying
  Proposition~\ref{prop:homotopical-comparison-product} in every
  simplicial degree and then using
  \cite[Lemma~8.7]{HSS-retractive} which
  states that the realization of a degree-wise local equivalence is a
  local equivalence. (This is one occasion were working over topological spaces would require an additional well-basedness hypothesis.)
\end{proof}

\begin{remark}\label{rem:action-of-F}
By~\cite[Proposition 2.27]{Sagave-S_group-compl}, the statement about $c_M$ already holds if $M$ is assumed to be semistable rather than positive $\cI$-fibrant. The conclusion of the proposition also holds for the action of $M = F = \GLoneIof{\bS}$ on $\bS$, as we show in~\cite[Proposition~9.17]{HSS-retractive} (see Example~\ref{ex:cart-I-monoids}(ii) for a definition). The extra input for that result is~\cite[Theorem 1.4]{Schlichtkrull_Thom-symmetric}. 
  
  We generally expect the present proposition to hold under a semistability assumption (instead of positive fibrancy) but since the present version suffices for our application, we refrain from working out the requisite generalization of~\cite[Theorem 1.4]{Schlichtkrull_Thom-symmetric}.
\end{remark}

\subsection{Thom spectra from cartesian actions}\label{proofs} Let $R$ be a positive fibrant commutative symmetric ring spectrum, let $\Omega^{\cI}(R)$ be the commutative $\cI$-space monoid model of its underlying multiplicative $E_{\infty}$ space, and let $\GLoneIof{R} \hookrightarrow \Omega^{\cI}(R)$ be the grouplike sub-commutative $\cI$-space monoid of units of $R$. The latter is formed by those components of the $(\Omega^{\cI}R)(\bld{n}) = \mathrm{Map}(S^n,R_n)$ that represent units in $\pi_0(\Omega^\cI(R))_{h\cI}$. Here the fibrancy condition on $R$ ensures that $\Omega^{\cI}(R)$ and  $\GLoneIof{R}$ capture a well defined homotopy type without having to be derived. It can be relaxed to asking $R$ to be semistable and level-fibrant (see~\cite[Remark 2.6]{Basu_SS_Thom}). 

We write $G$ for a cofibrant replacement of $\GLoneIof{R}$ in $\cC\cS^{\cI}$. By adjunction, the canonical map $G \to  \Omega^{\cI}(R)$ gives rise to a map of commutative parametrized symmetric ring spectra $\bS^{\cI}_t[G] \to R$. The \emph{universal $R$-line bundle} $\gamma_R$ is then defined to be $B^{\barsm}(\bS,\bS^{\cI}_t[G],R)^{\mathrm{fib}}$, a chosen fibrant replacement of the two-sided bar $B^{\barsm}(\bS,\bS^{\cI}_t[G],R)$ in the category of $BG$-relative commutative parametrized symmetric ring spectra (see~\cite[Definition~8.2]{HSS-retractive}).  

Now suppose that $M$ is a grouplike $\cI$-commutative cartesian $\cI$-monoid, and that $M$ acts on a positive fibrant commutative symmetric ring spectrum $R$. Then the adjoint $\mu^{\sharp}\colon M \to \Omega^{\cI}(R)$ of the action from~\eqref{eq:mu-sharp} factors through the inclusion of the units $\GLoneIof{R} \to \Omega^{\cI}(R)$. Applying the cofibrant replacement functor $(-)^{\mathrm{cof}}$ in commutative $\cI$-space monoids thus provides a map $\mu^{\mathrm{cof}}\colon M^{\mathrm{cof}}\to G = (\GLoneIof{R})^{\mathrm{cof}}$ to the cofibrant replacement of the units.  

\begin{theorem}\label{comparison theorem}
  Let $M$ be a positive fibrant grouplike $\cI$-commutative cartesian
  $\cI$-monoid acting on a positive fibrant commutative symmetric ring
  spectrum $R$ and an $R$-module $E$ such that both actions are
  Eckmann--Hilton. Then there is the following zig-zag of local
  equivalences of $ B^{\barsm}(\bS,\bS^{\cI}_t[M^{\mathrm{cof}}],R)$-module spectra in
  $\SpsymR$:
\[
(B^{\boxtimes}(\mu^{\mathrm{cof}}))^* B^{\barsm}(\bS,\bS^{\cI}_t[G],E)^{\mathrm{fib}} \ot B^{\barsm}(\bS,\bS^{\cI}_t[M^{\mathrm{cof}}],E) \to B^{\times}(\ucI,M,E) = E\sslash M\]
\end{theorem}
\begin{proof}
Since $M^{\mathrm{cof}}$ is flat,~\cite[Proposition~8.5]{HSS-retractive} provides the left hand local equivalence. The right hand map results from Proposition~\ref{lem:BSIMR-to-RsslashM-commutative} and is a local equivalence by Proposition~\ref{prop:bar-comparison}.
\end{proof}
When $E = R$, the maps in the theorem are commutative monoid maps, and the left hand side is simply the pullback of $\gamma_R$ along $B^{\boxtimes}(\mu^{\mathrm{cof}}) \colon B^{\boxtimes}(M^{\mathrm{cof}}) \to B^{\boxtimes}(G)$. Combined with Lemma~\ref{lem:BSIMR-to-RsslashM-commutative}, we therefore get:

\begin{corollary}\label{comparison corollary}
Let $R$ be a positive fibrant commutative symmetric ring spectrum and let $M$ be a  positive fibrant grouplike $\cI$-commutative cartesian $\cI$-monoid acting on $R$ with an Eckmann--Hilton action. Then we obtain a zig-zag
\[ (B^{\boxtimes}(\mu^{\mathrm{cof}}))^*\! (\gamma_R) \ot B^{\barsm}(\bS,\bS^{\cI}_t[M^{\mathrm{cof}}],R) \to R\sslash M\] 
of local equivalences of commutative $B^{\barsm}(\bS,\bS^{\cI}_t[M^{\mathrm{cof}}],R)$-algebra spectra in $\SpsymR$. \qed 
\end{corollary}
It is shown in~\cite[Section~9]{HSS-retractive} that 
$\Theta( (B\mu^{\mathrm{cof}})^* (\gamma_R))$ is a model for the Thom spectrum of $B(\mu^{\mathrm{cof}}) \colon B^{\boxtimes}(M^{\mathrm{cof}}) \to B^{\boxtimes}(G)$. Hence the corollary implies that ${\Theta(R\sslash M)}$ is equivalent to this Thom spectrum as a commutative symmetric ring spectrum. 

Finally, we detail the additions necessary to cover the topological case.
If $M$ is a cartesian $\cI$-monoid in topological spaces acting on a commutative ring spectrum $R$ in topological symmetric spectra, then geometric realization and singular complex induce a comparison map $| B^{\times}(*,\Sing(M), \Sing(R))| \to B^{\times}(*,M, R)$ since $|\!-\!|$ is strong symmetric monoidal and commutes with realization. Since $(|\!-\!|,\Sing)$ is a Quillen equivalence, the observation that the simplicial object $B^{\times}(*,M, R)$ is good in each $\cI$-degree when the $M(\bld{k})$ are well-based implies the following statement:
\begin{lemma}\label{lem: simp-top}
If $M$ is well-based in each $\cI$-degree, then the two comparison maps $| B^{\times}(*,\Sing(M), \Sing(R))| \to B^{\times}(*,|\Sing(M)|, R) \to B^{\times}(*,M, R)$ are both level equivalences in $\SpsymR$. \qed 
\end{lemma}
Together with \cite[Proposition~6.14]{HSS-retractive} this provides the following topological version of Corollary~\ref{comparison corollary} that again allows us to identify $R\sslash M$ as a Thom spectrum:
\begin{corollary}\label{cor:top-comparison-cor}
If we are working over topological spaces, $R$ and $M$ satisfy the assumptions of Corollary~\ref{comparison corollary}, and $M$ is well-based in each $\cI$-degree, then $R\sslash M$ is locally equivalent to $ |(B^{\boxtimes}(\Sing(\mu^{\mathrm{cof}})))^*\! (\gamma_{\Sing(R)})|$.  
\end{corollary}

\section{Examples of \texorpdfstring{$\cI$}{I}-monoid actions}\label{sec:examples-of-actions}
Our principal examples of cartesian $\cI$-monoids all arise from the following construction:

\begin{construction}\label{cartmon}
Let $(\cC,\tensor,\bld{1})$ be a symmetric monoidal category  enriched over the category of spaces $(\cS,\times,*)$. 

Given objects $C,D$ of $\cC$, we can define a cartesian $\cI$-monoid $\mathrm{End}_{\cI}(C;D)$ by setting $\mathrm{End}_{\cI}(C;D)(\bld{n}) = \mathrm{End}(D \tensor C^{\tensor \bld{n}})$ with monoid multiplication the composition. To give its structure maps, we first notice that $\alpha\colon \bld{m} \to \bld{n}$ in $\cI$ induces a bijection $\bld{m}\concat (\bld{n}-\alpha) \to \bld{n}$ where $\bld{n}-\alpha = \bld{n}\setminus \alpha(\bld{m})$ denotes the complement of the image of $\alpha$. This bijection induces an  isomorphism $\bar\alpha\colon C^{\bld{m}}\tensor C^{\bld{n}-\alpha} \to C^{\bld{n}}$ in $\cC$. Using $\bar\alpha$ and the identity in $\mathrm{End}(C^{\tensor \bld{n}-\alpha})$
we define $\alpha_*$ to be the composite 
\[\mathrm{End}(D \tensor C^{\tensor \bld{m}}) \xrightarrow{-\tensor \id} \mathrm{End}(D \tensor C^{\tensor \bld{m}} \tensor C^{\tensor \bld{n}-\alpha}) \xrightarrow{(\mathrm{id} \otimes \bar \alpha)^* \circ (\mathrm{id} \otimes \bar\alpha)_*}  \mathrm{End}(D \otimes C^{\tensor \bld{n}}).\]
The cartesian $\cI$-monoid  $\mathrm{End}_{\cI}(C;D)$ contains the cartesian $\cI$-monoid $\mathrm{Aut}_\cI(C;D)$ of invertible elements as a subobject. 
\end{construction}

We shall abbreviate $\mathrm{End}_{\cI}(C;\bld 1)$ to $\mathrm{End}_{\cI}(C)$ and similarly for $\mathrm{Aut}_\cI$. 

\begin{lemma}
The cartesian $\cI$-monoids  $\mathrm{End}_{\cI}(C)$ and  $\mathrm{Aut}_\cI(C)$ arising from the previous construction are $\cI$-commutative. \qed
\end{lemma}

\subsection{Examples of cartesian \texorpdfstring{$\cI$}{I}-monoids} We start out with immediate examples of Construction \ref{cartmon}.

\begin{example}\phantomsection \label{ex:cart-I-monoids}
\begin{enumerate}[(i)]
\item A few pertinent examples include the cartesian $\cI$-monoids $\Sigma_\cI$, given as $\mathrm{Aut}_\cI(*)$ taken in the category of sets under disjoint union, $\capitalGL_\cI(R) = \mathrm{Aut}_\cI(R)$ taken in the category of modules over a ring $R$ under direct sum, or $\mathrm{O}_\cI$ and $\mathrm{U}_\cI$ given by $\mathrm{Aut}_\cI(\mathbb K)$ in the category of inner product spaces over $\mathbb K = \mathbb R$ or $\mathbb C$ under direct sum. Similarly, the category of oriented finite dimensional inner product spaces leads to $\mathrm{SO}_\cI$ and the category of topological spaces under cartesian product yields $\mathrm{Top}_\cI$.
\item The $\cI$-space $\Omega^{\cI}(\mathbb S)$ is a cartesian $\cI$-monoid since
	$\Omega^{\cI}(\mathbb S) = \mathrm{End}_\cI(S^1)$ 
	in the category $\cS_*$, monoidal under the smash product. The monoid $\Omega^\cI(\mathbb S)$ acts on the sphere spectrum $\mathbb S$ by the evaluation maps
	$\mathrm{ev} \colon \Omega^nS^n \times S^n \rightarrow S^n$,
	and the map 
	$\mathrm{ev}^\sharp \colon \Omega^\cI(\mathbb S) \to \Omega^\cI(\mathbb S)$
	from Lemma \ref{dp-map} is the identity. 
\item Using the functoriality of universal covers (for locally contractible connected groups, say), the example $\mathrm{SO}_\cI$ gives rise to $\mathrm{Spin}_\cI$. It is again commutative.
\end{enumerate} 
\end{example}
For our treatment of the Spin-bordism spectrum below, we also need the example $\mathrm{Pin}_\cI$. The Pin-groups, however, do \emph{not} form an $\cI$-commutative cartesian $\cI$-monoid since odd elements of $\mathrm{Pin}(n)$ and $\mathrm{Pin}(k)$ anticommute in $\mathrm{Pin}(n+k)$. Therefore, they cannot arise from the construction above. To capture their commutativity we need to regard them as \emph{$\cI$-supergroups} as in \cite{Stolz-conc}. Let us more generally define supercategories, a linear version of which appears for example in \cite{supercat}.

To this end consider the category of superspaces, i.e., $\bZ^*$-spaces equipped with an invariant map $|\cdot|$ to $\bZ/2$ called the grading. (We use $\bZ^* = \{\pm 1\}$ instead of $\bZ/2$ for the group action to keep the two structures apart.) This category is symmetric monoidal under the superproduct: For two objects $X,Y$ put $X \widehat{\times} Y = X \times_{\bZ^*} Y$ equipped with the residual $\bZ^*$-action and the grading induced by the composite
\[X \times Y \rightarrow \bZ/2 \times \bZ/2 \xrightarrow{+} \bZ/2.\]
The unit is $\bZ^*$ with free action and trivial grading. The non-trivial braiding
\[X \widehat{\times} Y \longrightarrow Y \widehat{\times} X, \quad [x,y] \longmapsto  [(-1)^{|x||y|}y,x] = [y,(-1)^{|x||y|}x]\]
provides the symmetry isomorphism. 

\begin{definition}
A supercategory is a category enriched in superspaces. The enriched product of two supercategories we shall refer to as their superproduct. A (symmetric) monoidal supercategory is just an enriched (symmetric) monoidal supercategory.
\end{definition}

\begin{observation}
The category of $\bZ/2$-graded vector spaces over some field, symmetric monoidal under the tensor product and Koszul-braiding, admits the forgetful functor $V \mapsto V^{\mathrm{ev}} \sqcup V^{\mathrm{odd}}$, with the action of $\bZ^*$ through $-1$ in the base field. This functor admits a tautological lax symmetric monoidal structure. Without further assumptions one has to regard $V^{\mathrm{ev/odd}}$ as a discrete space, but clearly when considering some category of topological vector spaces with a well-behaved tensor product this can promoted to take the topology into account.

In particular, every linear (symmetric) monoidal supercategory as considered in \cite{supercat}, has an underlying (symmetric) monoidal supercategory in our sense. 
\end{observation}

We can now define a supermonoid as a supercategory with one object. This unfolds exactly to the notion considered by Stolz in \cite{Stolz-conc}: It is a (topological) monoid $M$, together with a homomorphism $M \rightarrow \mathbb Z/2$, a central element $c \in M$ in degree $0$ with $c^2 = e$ (namely $-e$, which determines the entire action). A homomorphism of supermonoids is a monoid homomorphism carrying one distinguished element to the other and preserving the gradings. The superproduct of two supermonoids is clearly again a supermonoid. It is given by $M \widehat\times H$ with multiplication
\[[g,h]\cdot[g',h'] = [c^{|g'||h|}gg',hh'].\]
and new distinguished element $[c,e] = [e,d]$. 

In analogy with Definition~\ref{i-comm}, we set:

\begin{definition}
An $\cI$-supermonoid $M$ is $\cI$-commutative if for any injective map $\alpha \colon {\bld k \sqcup \bld l} \rightarrow \bld m$, the diagram 
\[\xymatrix@-1pc{ M(\bld{k}) \widehat{\times} M(\bld{l}) \ar[rrr]^{(\alpha|_{\bld{k}})_* \widehat{\times} (\alpha|_{\bld{l}})_*} \ar[d]_{\mathrm{twist}} & && M(\bld{m}) \widehat{\times} M(\bld{m}) \ar[rrd]^{\mu} & & \\
M(\bld{l})  \widehat{\times} M(\bld{k}) \ar[rrr]^{(\alpha|_{\bld{l}})_* \widehat{\times} (\alpha|_{\bld{k}})_*} & & & M(\bld{m}) \widehat{\times} M(\bld{m}) \ar[rr]^{\mu} & & M(\bld{m}).
}\]
commutes. 
\end{definition}

We emphasize that the left vertical map is the braiding morphism of superspaces and thus contains a sign. In particular, in general $\cI$-commutativity neither passes from an $\cI$-supermonoid $M$ to its underlying cartesian $\cI$-monoid nor vice versa unless the distinguished element $c$ is the unit. Given an $\cI$-supermonoid $G$ one can, however, form the cartesian $\cI$-monoid $PG$, the degreewise quotient by the distinguished central elements, and $G^{\mathrm{ev}}$ consisting of the even elements, both of which are $\cI$-commutative if $G$ is.

Examples of $\cI$-commutative $\cI$-supermonoids arise by an analogous construction as before:

\begin{construction}
Let $(\cC,\tensor,\bld{1})$ be a symmetric monoidal supercategory. For objects $C,D \in \cC$ we set $\mathrm{End}^s_{\cI}(C;D)(\bld{n}) = \mathrm{End}(D \tensor C^{\tensor \bld{n}})$ and let $\mathrm{Aut}^s_{\cI}(C;D)$ be the subobject of invertible elements. With the same structure maps as before, these define $\cI$-supermonoids.
\end{construction}

Abbreviating the case $D= \bld 1$ as before, it is readily checked that $\mathrm{End}^s_{\cI}(C)$ is again $\cI$-commutative.

\begin{example}
As announced the $\cI$-supermonoids $\mathrm{Pin}_\cI$ and $\mathrm{Pin}^c_\cI$ arise as a subobject of $\mathrm{End}^s_\cI(\mathrm{Cl}_1\mathbb K)$ in the supercategory of $\mathbb Z/2$-graded $\mathbb K$-vector spaces under the tensor product with Koszul-signed braiding as follows. Using the standard isomorphisms $\mathrm{Cl}(V) \otimes \mathrm{Cl}(W) \rightarrow \mathrm{Cl}(V \oplus W)$ 
we recover $\mathrm{Cl}_n\mathbb K$ as the right $\mathrm{Cl}_1(\mathbb K)^{\hat\otimes n}$-linear elements of $\mathrm{End}^s_\cI(\mathrm{Cl}_1\mathbb K)(\bld n)$. As usual the $\mathrm{Pin}$- and $\mathrm{Pin}^c$-groups arise as the subgroups thereof that furthermore stabilize $\mathbb K^n \subseteq \mathrm{Cl}_n\mathbb K \cong \mathrm{Cl}_1(\mathbb K)^{\hat\otimes n}$ under twisted conjugation.
\end{example}
\begin{example}\label{ex:OI-and-UI} 
  Consider the spaces $O(L^2(\mathbb R^n,\mathbb R))$ and $U(L^2(\mathbb R^n,\mathbb C))$ of even or odd orthogonal (resp. unitary) operators on the displayed Hilbert spaces, $\mathbb Z/2$-graded by even and odd functions. Note that the maps
  \[L^2(\mathbb R^m) \otimes L^2(\mathbb R^n) \to L^2(\mathbb R^{m+n}), \quad f \otimes g \mapsto \big((x,y) \mapsto f(x)\cdot g(y)\big)\] are isomorphisms. This implies that the spaces $O(L^2(\mathbb R^n,\mathbb R))$ and $U(L^2(\mathbb R^n,\mathbb C))$ are part of cartesian $\cI$-supermonoids $\mathrm{Aut}^s_\cI(L^2(\mathbb R,\mathbb K))$ for $\mathbb K = \mathbb R, \mathbb C$, formed in the category of $\bZ/2$-graded Hilbert spaces over $\mathbb K$ under tensor product. For later purposes, we here use the strong operator topology (often referred to as the point-norm topology in this context) rather than the norm topology. We shall call these cartesian $\cI$-supermonoids $\mathcal O_\cI$ and $\mathcal U_\cI$, respectively.  Their projectivizations will form the basis for our discussion of twisted $K$-theory and twisted Spin-cobordism. They were originally devised by Joachim and the second author in \cite{HeJo} for this purpose.
\end{example}
\begin{example}\label{j} For the definition of the underlying spectra we shall also need the $\cI$-supermonoids formed by the subgroups of right Clifford-linear isometries of
\[L^2(\mathbb R^n,\mathbb K) \otimes \mathrm{Cl}(\mathbb K^n). \] They form a subobject of the $\cI$-supermonoid $\mathrm{Aut}^s_\cI((L^2(\mathbb R,\mathbb K) \otimes \mathrm{Cl}(\mathbb K))$ taken in $\bZ/2$-graded Hilbert spaces under tensor product. We shall call them $\widehat{\mathcal O}_\cI$ and $\widehat{\mathcal U}_\cI$, respectively. As well as containing $\mathcal O_\cI$ or $\mathcal U_\cI$, they receive embeddings of $\cI$-supermonoids
\[j \colon \mathrm{Pin}_\cI \longrightarrow \widehat{\mathcal O}_\cI \quad \text{and} \quad j^c \colon \mathrm{Pin}^c_\cI \longrightarrow \widehat{\mathcal U}_\cI\] given by $p \longmapsto \big(f \otimes c \mapsto (-1)^{|p||f|} f \circ \rho(p)^{-1} \otimes p\cdot c\big)$ with $\rho$ the usual representation of the $\mathrm{Pin}$- and $\mathrm{Pin}^c$-groups on euclidean space (see \cite[Section 3.2]{HeJo} for a verification). These restrict to maps
\[j \colon \mathrm{Spin}_\cI \longrightarrow \widehat{\mathcal O}^{\mathrm{ev}}_\cI \quad \text{and} \quad j^c \colon \mathrm{Spin}^c_\cI \longrightarrow \widehat{\mathcal U}^{\mathrm{ev}}_\cI\]
on the even components for which we use the same name. 
\end{example}
\begin{example}\label{ex:POI} The projectivization of the construction in Example~\ref{ex:OI-and-UI} has a generalization to arbitrary $C^*$-algebras $A$: Using the category of (real or complex) $\mathbb Z/2$-graded $C^*$-algebras with the spatial tensor product, we can consider the $\cI$-monoid defined by $\mathrm{Aut}_{\cI}(A \otimes \cK)$. Here $\mathcal K$ denotes the $C^*$-algebra of graded compact operators on $L^2(\mathbb R,\mathbb K)$ and the automorphism space is again equipped with the strong operator (or point-norm) topology rather than the norm topology. These cartesian $\cI$-monoids were first considered by Dadarlat and Pennig in \cite{DP-I}. For the base fields we have canonical identifications
	\[\mathrm{Aut}_{\cI}(\cK) \cong P\mathcal O_\cI \quad \text{and} \quad \mathrm{Aut}_{\cI}( \cK) \cong P_{S^1}\mathcal U_\cI,\] in the real and complex case, respectively. Here the subscript in the complex case indicates projectivization by dividing out $S^1$ instead of just $-1$, and the identification is induced by the isomorphisms
	\[\mathcal K(H)^{\otimes n} \cong \mathcal K(H^{\otimes n}) \quad \text{and}\quad \mathrm{Aut}(\mathcal K(H)) \cong PO(H) \text{ or } P_{S^1}U(H)\] for any Hilbert space $H$ given by tensoring operators and conjugation, respectively. 
      \end{example}
      \begin{example} Similar to Example~\ref{j}, we shall also need to consider the two variants \[ \mathrm{Aut}_\mathcal I(A \otimes \mathcal K) \longrightarrow \mathrm{Aut}_\mathcal I(A \otimes \mathcal K, A) \longleftarrow \mathrm{Aut}_\mathcal I(\mathcal K,A)\] to compare the outer two terms of the previous example for certain $C^*$-algebras later.
\end{example}

\section{Spin bordism spectra and their twists}\label{sec:spin}
In this section we give a general construction of tautological twists of Thom spectra and apply it to study twists of spin bordism spectra. Furthermore, we establish an Eckmann-Hilton action of the cartesian $\cI$-monoid $\mathrm P \mathcal O_\cI$ on the Spin bordism spectrum. This will serve as a crucial ingredient into the proof of the main theorem.

\subsection{Tautological twists of Thom spectra}\label{section:taut-twists} Let $(R,M)$ be a commutative para\-metrized ring spectrum, let $P$ be a grouplike  commutative $\cI$-space monoid, and let $f \colon \bS^{\cI}_t[P] \rightarrow (R,M)$ be a map of commutative parametrized ring spectra. Applying $\Theta \colon \SpsymR \rightarrow \Spsym{}$ to $f$ provides map of ring spectra
$\bS^{\cI}[P] \longrightarrow \Theta(R,M).$
Composing it with a fibrant replacement in $\cC\Spsym{}$ and using the adjunction $(\bS^{\cI},\Omega^{\cI})$ we get a map 
$P \to \Omega^{\cI}(\Theta(R,M)^{\mathrm{fib}})$ that factors through  
\[ \hat{f}\colon P \to  \GLoneIof(\Theta(R,M)^{\mathrm{fib}}) \]
since $P$ is grouplike.

One standard example of this construction arises as follows: Let $O_{\cI}$ be the $\cI$-commutative cartesian $\cI$-monoid $\bld{n} \mapsto O(n)$ given by the orthogonal groups (see Example~\ref{ex:cart-I-monoids}(i) above), let $H$ be  an $\cI$-commutative cartesian $\cI$-monoid, and let $\theta \colon H \rightarrow O_\cI$ be a map of cartesian $\cI$-monoids. We then consider $B^\times(*,H,\mathrm{O}_\cI)$, the homotopy quotient of the $H$-action on $\mathrm{O}_\cI$. By a slight generalization of Lemma~\ref{lem:commutative-bar-constr-comp-for-M}, it inherits a commutative $\cI$-space monoid structure from $H$ and $\mathrm{O}_\cI$. 

The commutative $\cI$-space monoid $B^{\times}(\mathrm{O}_\cI)$ is the base space of the parametrized spectrum $\gamma^+ = B^\times(*,\mathrm{O}_\cI,\bS)$.  Assuming in addition that  $P = B^\times(*,H,\mathrm{O}_\cI)$ is grouplike, we can then apply the above construction with $(R,M) = (B\theta)^*\gamma^+ \iso B^\times(*,H,\bS)$ and $f$ the composite 
\[\xymatrix{\bS^{\cI}_t[B^\times(*,H,\mathrm{O}_\cI)] \ar[r]^-\cong & B^\times(*,H,\bS^{\cI}_t[O_\cI]) \ar[rr]^-{B^{\times}(*,\mathrm{id},\mathrm{act})} && B^\times(*,H,\bS)} \]
where $\mathrm{act} \colon \bS^{\cI}_t[O_{\cI}] \to \bS$ is the adjoint~\eqref{eq:mu-flat} of the canonical $O_{\cI}$-action on $\bS$. 
In this case we follow the convention of writing ${M}\theta$ for $\Theta(R,M)^{\mathrm{fib}}$. The map 
\[\taut_\theta = \hat{f} \colon B^\times(*,H,\mathrm{O}_\cI) \longrightarrow \GLoneIof{{M}\theta}\]
is our model of the inclusion of the \emph{lower} or {tautological twists} of the Thom spectrum $M\theta$. Applying the naturality of this construction in $H$ to the map $* \to H$ provides a commutative diagram of commutative $\cI$-space monoids 
\[\xymatrix@-1pc{\mathrm{O}_\cI \ar[rr] \ar[d]_{J} && B^\times(*,H,\mathrm{O}_\cI) \ar[d]^{\taut_\theta} \\
  \GLoneIof{\bS} \ar[rr] && \GLoneIof{{M}\theta}.}\]
Applying the bar construction $B$ with respect to ${\boxtimes}$ to (a cofibrant replacement of) this square gives an analogous commutative square of classifying spaces. 

\begin{remark}
Assuming that $H$ is an $\cI$-space monoid (not necessarily cartesian) and replacing all instances of the $\times$- with the $\boxtimes$-products above one obtains a similar construction without assuming $H$ to be commutative. However, for later considerations the present construction is more convenient, so we refrain from carrying this out.
\end{remark}

\subsection{\texorpdfstring{$\mathrm{MSpin}$}{MSpin} and \texorpdfstring{$\mathrm{MSpin}^c$}{MSpin^c}}
We now review the definition of Joachim's operator algebraic models for the $\mathrm{Spin}$- and $\mathrm{Spin}^c$-bordism spectra \cite[Section~6]{Jo-coherence} and explain how the $\cI$-commutative cartesian $\cI$-monoids $P\mathcal O_\cI$ and $P_{S^1}\mathcal U_\cI$ from Example~\ref{ex:POI} act on these spectra.

In our setup, Joachim's models can be constructed from the parametrized spectra
\[\widehat{\mathcal O}^{\mathrm{ev}}_\cI \times_{\mathrm{Spin}_\cI} \bS \quad \text{and} \quad \widehat{\mathcal U}^{\mathrm{ev}}_\cI \times_{\mathrm{Spin}^c_\cI} \bS, \] which are given as the coequalizer of two different action maps
\[\widehat{\mathcal O}^{\mathrm{ev}}_\cI \times {\mathrm{Spin}_\cI} \times \bS \rightrightarrows \widehat{\mathcal O}^{\mathrm{ev}}_\cI \times \bS \quad \text{and} \quad \widehat{\mathcal U}^{\mathrm{ev}}_\cI \times {\mathrm{Spin}^c_\cI} \times \bS \rightrightarrows \widehat{\mathcal U}^{\mathrm{ev}}_\cI \times \bS.\] Here the action on the left factor being via the embeddings $j$ and $j^c$ from Example~\ref{j}, and that on the right factor via the morphisms of cartesian $\cI$-monoids
\[\mathrm{Pin}_\cI \rightarrow \mathrm{Pin}^c_\cI \longrightarrow O_\cI \longrightarrow \GLoneIof{\bS}.\]
 As the action of $\mathrm{Pin}^c(n)$ on $S^n$ factors through $O(n)$ we find the $n$th level of these spectra can also be described as
\begin{equation}\label{eq:P-quotients} ({P}\widehat{\mathcal O}_n\times_{O(n)}S^n, P\widehat{\mathcal O}_n/j{O}(n))\quad \text{and} \quad ({P}_{S^1}\widehat{\mathcal U}_n\times_{O(n)}S^n, P_{S^1}\widehat{\mathcal U}_n/j^c{O}(n)).
\end{equation}
We then set\[\begin{split} \mathrm{MSpin} &=  \Theta({P}\widehat{\mathcal O}_\cI \times_{O_\cI} \bS) \cong \Theta(\widehat{\mathcal O}^{\mathrm{ev}}_\cI \times_{\mathrm{Spin}_\cI} \bS)\quad \text{and}\\ \mathrm{MSpin}^c &= \Theta(P_{S_1}\widehat{\mathcal U}_{\cI}\times_{O_{\cI}} \bS) \cong \Theta(\widehat{\mathcal U}^{\mathrm{ev}}_\cI \times_{\mathrm{Spin}^c_\cI} \bS). 
  \end{split}
\]
Let us restrict the discussion to the case of $\mathrm{MSpin}$, the complex case being entirely analogous.

We note that the base section $P\widehat{\mathcal O}_n/j{O}(n) \rightarrow {P}\widehat{\mathcal O}_n\times_{O(n)}S^n$ is a cofibration: By \cite[Satz~3.13 and Satz~3.26]{tDKaPu70} it suffices to show that its image is given as the vanishing locus of a real valued function and admits a halo that contracts onto it. The three required pieces of data are all readily constructed from the metric given through the structure group.

The parametrized spectrum $\widehat{\mathcal O}^{\mathrm{ev}}_\cI \times_{\mathrm{Spin}_\cI} \bS$ and thus also $\mathrm{MSpin}$ in fact form commutative ring spectra via the concatenation operation $\widehat{\mathcal O}^{\mathrm{ev}}_n \times \widehat{\mathcal O}^{\mathrm{ev}}_m \longrightarrow \widehat{\mathcal O}^{\mathrm{ev}}_{n+m}$ arising from the $\cI$-monoid structure.  The action maps $\mu_n\colon P\mathcal O_n \times \mathrm{MSpin}_n \rightarrow \mathrm{MSpin}_n$ are now simply given by left multiplication on the left factor. It is straightforward to check that they are $\mathrm{MSpin}$-linear.
	
We now claim that the two maps 
\[\mu^\sharp \colon {P}{\mathcal O}_\cI \longrightarrow \GLoneIof{\mathrm{MSpin}} \quad \text{and} \quad \mu^\sharp \colon {P}_{S^1}{\mathcal U}_\cI \longrightarrow \GLoneIof{\mathrm{MSpin}^c}\]
induced from~\eqref{eq:mu-sharp} are in fact equivalent to the tautological ones from Section~\ref{section:taut-twists}. To see this, we observe that there is a commutative diagram
\begin{equation}\label{eq:spin-comparison}\xymatrix@-1pc{\bS_t^\cI[O_\cI \sslash \mathrm{Spin}_\cI] \ar[r] \ar[d]^{\bS_t^\cI[j \sslash j]}& B^\times(*,\mathrm{Spin}_\cI,\bS) \ar[r] \ar[d]^{B^{\times}(*,j,\mathrm{incl})} &  B^\times(*,{O}_\cI,\bS) \ar[d]^{B^{\times}(*,j,\mathrm{incl})}\\
            \bS^\cI_t[P\widehat{\mathcal O}_\cI \sslash \widehat{\mathcal{O}}^{\mathrm{ev}}_\cI] \ar[r]& B^\times(*,\widehat{\mathcal{O}}^{\mathrm{ev}}_\cI,P\widehat{\mathcal O}_\cI \times_{O_\cI} \bS)\ar[r] & B^\times(*,P\widehat{\mathcal O}_\cI,P\widehat{\mathcal O}_\cI\times_{O_\cI} \bS) \\
            \bS_t^\cI[P\mathcal O_\cI] \ar[u] \ar[r]^{\mathrm{proj}}&  {P}\widehat{\mathcal O}_\cI \times_{O_\cI} \bS \ar[r]\ar[u]&  B^\times(*,P\widehat{\mathcal O}_\cI,\mathrm{P}\widehat{\mathcal O}_\cI\times_{O_\cI}\bS). \ar@{=}[u]}
\end{equation}
        with the upper two left horizontal maps given by
\[\bS_t^\cI[O_\cI \sslash \mathrm{Spin}_\cI] \cong B^\times(*,\mathrm{Spin}_\cI,\bS^\cI_t[{O}_\cI]) \xrightarrow{B^{\times}(*,\mathrm{id},\mathrm{act})} B^\times(*,\mathrm{Spin}_\cI,\bS)\]
and 
\[\bS^\cI_t[P\widehat{\mathcal O}_\cI \sslash \widehat{\mathcal{O}}^{\mathrm{ev}}_\cI] \cong B^\times(*,\widehat{\mathcal{O}}^{\mathrm{ev}}_\cI,\bS_t^\cI[P\widehat{\mathcal O}_\cI]) \xrightarrow{B^{\times}(*,\mathrm{id},\mathrm{proj})} B^\times(*,\widehat{\mathcal{O}}^{\mathrm{ev}}_\cI,P\widehat{\mathcal O}_\cI \times_{O_\cI} \bS).\]
Note also that the upper middle and right terms are just the fiberwise one-point-compactifications of the universal vector bundles over $B^{\times}(\mathrm{Spin}_\cI)$ and $B^{\times}(O_\cI)$, respectively.
 
\begin{proposition}\label{prop-good}
The simplicial parametrized spectra in~\eqref{eq:spin-comparison} are all levelwise good (i.e., their degeneracies are cofibrations) and all vertical maps are level equivalences of parametrized spectra. In particular, all horizontal sequences are levelwise homotopy fiber sequences. 
\end{proposition}            
           
\begin{proof}
To see the goodness assertion we need only show that that the groups $P\widehat{\mathcal{O}}^{\mathrm{ev}}_n$ are well-pointed. For $n = 0$ there is nothing to do, so suppose $n>0$. In the topology under consideration it follows for example from \cite[Proposition 2.26]{DP-III} (for $A$ the base field), and its real and quaternionic analogues, that the groups $PO(H), P_{S^1}U(H)$ and $P_{S^3}Sp(H)$ for a real, complex or quaternionic Hilbert space $H$, respectively, are all well-pointed. But by Schur's lemma $\widehat{\mathcal{O}}^{\mathrm{ev}}_n$ is isomorphic to a product of at most two factors of the form $O(H), U(H)$ or $Sp(H)$, depending on the representation type of the Clifford algebra; the irreducible representations appear infinitely often in $L^2(\mathbb R^n,\mathbb R) \otimes \mathrm{Cl}(\mathbb R^n)$ as the left factor contains an infinite dimensional summand with trivial $\mathrm{Cl}(\mathbb R^n)$-action, namely the radially symmetric functions. Therefore $P\widehat{\mathcal{O}}^{\mathrm{ev}}_n$ is at most $2$-sheeted cover over a well-pointed space and thus well-pointed itself.

Now the inclusion $j\colon \mathrm{Spin}_\cI \!\to \widehat{\mathcal O}^{\mathrm{ev}}_\cI$ induces an isomorphism $\widehat{\mathcal O}^{\mathrm{ev}}_\cI \!\times_{\mathrm{Spin}_\cI}\! O_\cI \to P\widehat{\mathcal O}_\cI$. Using this isomorphism, the left most column can be rewritten as the fiberwise suspension spectra of the maps
\[E\mathrm{Spin}_\cI \times_{\mathrm{Spin}_\cI} O_\cI \rightarrow E\widehat{\mathcal O}^{\mathrm{ev}}_\cI \times_{\mathrm{Spin}_\cI} O_\cI \cong E\widehat{\mathcal O}^{\mathrm{ev}}_\cI \times_{\widehat{\mathcal O}^{\mathrm{ev}}_\cI} P\widehat{\mathcal O}_\cI \leftarrow P\widehat{\mathcal O}_\cI\]
with the outer maps induced by $j$ and inclusions. The left map is evidently an equivalence and for the right one it follows immediately from Kuiper's theorem, that $\widehat{\mathcal O}^{\mathrm{ev}}_n$ is contractible (this holds in the strong instead of the norm topology by \cite[Proposition 10.8.2]{Dixmier-book}).

For the middle and right column the total spaces are also sphere bundles, albeit no longer trivial. We can rewrite the total spaces of these columns as 
\[E\mathrm{Spin}_\cI \times_{\mathrm{Spin}_\cI} \bS \rightarrow E\widehat{\mathcal O}^{\mathrm{ev}}_\cI \times_{\mathrm{Spin}_\cI} \bS \cong E\widehat{\mathcal O}^{\mathrm{ev}}_\cI \times_{\widehat{\mathcal O}^{\mathrm{ev}}_\cI} P\widehat{\mathcal O}_\cI\times_{O_\cI} \bS \leftarrow P\widehat{\mathcal O}_\cI\times_{O_\cI} \bS\]
and $EO_\cI \times_{O_\cI} \bS \to EP\widehat{\mathcal O}_\cI\times_{O_\cI} \bS$. Then the same reasoning applies. Clearly, all of these equivalences are compatible with the structure maps. The statement about the rows follows since it is true for the top one.
\end{proof}
In the sequel, we shall write $M\mathrm{Spin} =  \Theta B^\times(*,\mathrm{Spin}_\cI,\bS)$ for the standard model of the Spin-bordism spectrum. We stress that the font of the letter $M$ distinguishes it notationally from Joachim's model $\mathrm{MSpin} = \Theta({P}\widehat{\mathcal O}_\cI \times_{O_\cI} \bS)$ considered earlier.  The last proposition implies:
\begin{corollary}\label{cor:Mspin-comparison} There is a zig-zag of level equivalences between commutative symmetric ring spectra relating $M\mathrm{Spin}$ and $\mathrm{MSpin}$. \qed
\end{corollary}
\begin{remark} This comparison seems to not occur in the literature so far. In particular, it shows that Joachim's spectrum $\mathrm{MSpin}$ carries the `correct' $E_{\infty}$ structure.
\end{remark}

We now apply the construction of Section \ref{section:taut-twists} to the left hand horizontal maps in~\eqref{eq:spin-comparison}. For this we note that we can drop the fibrant replacements of the spectra  $M\mathrm{Spin}$ and $\mathrm{MSpin}$ since these result from orthogonal spectra and are thus semistable. The construction provides a diagram of commutative $\cI$-space monoids
\begin{equation}\label{taut vs DK}
\xymatrix@-1pc{\mathrm{O}_\cI \sslash \mathrm{Spin}_\cI \ar[r]^{\taut_\mathrm{Spin}} \ar[d] & \GLoneIof{M\mathrm{Spin}} \ar[d]\\
P\widehat{\mathcal O}_\cI \sslash \widehat{\mathcal{O}}^{\mathrm{ev}}_\cI \ar[r] & \GLoneIof{(\Theta B^\times(*,\widehat{\mathcal{O}}^{\mathrm{ev}}_\cI,P\widehat{\mathcal O}_\cI \times_{O_\cI} \bS))^{\mathrm{fib}}} \\
P\mathcal O_\cI \ar[r]^{\mathrm{proj}^\ttt}\ar[u] & \GLoneIof{\mathrm{MSpin}}. \ar[u]}
\end{equation} 
Its vertical maps are level equivalences by Proposition~\ref{prop-good}. Let $\mu \colon P\mathcal O_\cI \times \mathrm{MSpin} \rightarrow \mathrm{MSpin}$ be the action constructed above. Unwinding definitions we now find that $\mathrm{proj}^\ttt$ equals the  map $\mu^\sharp$ from~\eqref{eq:mu-flat}. This equality is ultimately the heart of our comparison of operator algebraic and homotopical twisted $K$-theory. For now, denoting by $(-)^{\mathrm{cof}}$ cofibrant replacements of the maps above and suppressing the passage to singular complexes  from the notation, we obtain:

\begin{corollary}\label{comparison MSpin}
There are zig-zags of local equivalences  
\begin{align*} (B^{\boxtimes}(\taut_{\mathrm{Spin}}^{\mathrm{cof}}))^*\gamma_{M\mathrm{Spin}} & \simeq (B^{\boxtimes}(\mu^{\mathrm{cof}}))^*\gamma_\mathrm{MSpin} \quad \text{and}\\ (B^{\boxtimes}(\taut^{\mathrm{cof}}_{\mathrm{Spin}^c}))^*\gamma_{M\mathrm{Spin}^c} &\simeq (B^{\boxtimes}(\mu^{\mathrm{cof}}))^*\gamma_\mathrm{MSpin^c}
\end{align*}
of commutative parametrized ring spectra. \qed
\end{corollary}

\begin{remark}
We also obtain a comparison zig-zag between $(B^{\boxtimes}(\mu^{\mathrm{cof}}))^*\gamma_\mathrm{MSpin}$ and $\mathrm{MSpin} \sslash P\mathcal O_\cI$ from \eqref{eq:bar-constr-comparison}. It seems likely that it consists of local equivalences as well. Since $\mathrm{MSpin}$ is far from being positive fibrant as a symmetric spectrum, this is, however, not implied by Corollary \ref{comparison corollary}. As the above corollary suffices for the analogous comparison in $K$-theory, we have refrained from working out further details.
\end{remark}

\section{Models of twisted \texorpdfstring{$K$}{K}-theory spectra}\label{finalsec}
We are now ready to establish our models for twisted $K$-theory spectra. In particular, we establish Theorem~\ref{thm:d} in Section~\ref{comp-geo-homo} and Theorem~\ref{thm:e} in Corollaries~\ref{tata} and~\ref{tata2}.

\subsection{Comparison of pairings in twisted  \texorpdfstring{$K$}{K}-theory} We work in the setup of Example~\ref{ex:POI}. For a $C^*$-algebra $A$ the $\cI$-monoid $\mathrm{Aut}_\cI(\mathcal K,A)$ acts on the symmetric $K$-theory spectrum $\mathrm K_A$ of \cite[Definition 4.9]{Jo-kspectra} given by
 \[(\mathrm K_A)_\bld n = \mathrm{Hom}_{*}(\widehat{\mathrm S}, A \otimes \mathcal K(L^2(\mathbb R^n)) \otimes \mathrm{Cl}_n)\]
through post composition; here $\widehat{\mathrm S}$ denotes the bi-algebra $C_0(\mathbb R)$ of continuous functions vanishing at infinity. Both $\mathrm{Aut}_\cI(\mathcal K,A)$ and $\mathrm K_A$ are positive fibrant for any $A$: For $\mathrm{Aut}_\cI(\mathcal K,A)$ this follows immediately from the existence of an isomorphism $\mathcal K \otimes \mathcal K \cong \mathcal K$, conjugation with which is homotopic to $- \otimes \mathrm{id}_\mathcal K$, and for $\mathrm K_A$ this is proved in~\cite[Theorem~4.10]{Jo-kspectra}. 

However, the groups $\mathrm{Aut}_\cI(\mathcal K,A)$ are not generally well-pointed, for example when $A = C(X)$ for some pathological compact Hausdorff space $X$, so we shall replace them by the realizations $s\mathrm{Aut}_\cI(\mathcal K,A)$ of their singular complexes throughout. Whenever $\mathrm{Aut}_\cI(\mathcal K,A)$ happens to be well-pointed, as in the case $A = \mathbb K$ or more generally in the case of strongly selfabsorbing algebras below, the spectra $\mathrm K_A \sslash \mathrm{Aut}_\cI(\mathcal K,A)$ and $\mathrm K_A \sslash s\mathrm{Aut}_\cI(\mathcal K,A)$ are locally equivalent by Lemma \ref{lem: simp-top} and no construction we make will be sensitive to the change. 

Now the coalgebra structure of the suspension algebra provides pairings
 \[\big(\mathrm K_A \sslash s\mathrm{Aut}_\cI(\mathcal K,A)\big) \barsm \big(\mathrm K_B \sslash s\mathrm{Aut}_\cI(\mathcal K,B)\big) \longrightarrow \mathrm K_{A \otimes B} \sslash s\mathrm{Aut}_\cI(\mathcal K,A \otimes B)\]
 by tensoring homomorphisms: They are induced by
\begin{align*}
(s\mathrm{Aut}_\cI(\mathcal K,A)^{\times q} \times \mathrm K_A) &\barsm  (s\mathrm{Aut}_\cI(\mathcal K,B)^{\times q} \times \mathrm K_B) \\
&\longrightarrow (s\mathrm{Aut}_\cI(\mathcal K,A) \boxtimes s\mathrm{Aut}_\cI(\mathcal K,B))^{\times q} \times (\mathrm K_A \sm \mathrm K_B) \\
&\longrightarrow s\mathrm{Aut}_\cI(\mathcal K,A \otimes B)^{\times q} \times \mathrm K_{A \otimes B}
\end{align*}
with the first map an instance of the distributivity morphism~\eqref{eq:distr} and the second map induced by the evident product maps.
In particular, for $A$ the base field we obtain commutative parametrized ring spectra 
 \[\mathrm{KO} \sslash P\mathcal O_n \simeq \mathrm K_\mathbb R \sslash s\mathrm{Aut}_\cI(\mathcal K)  \quad \text{and} \quad \mathrm{KU} \sslash P_{S^1}\mathcal U_n \simeq \mathrm K_\mathbb C \sslash s\mathrm{Aut}_\cI(\mathcal K)\]
 over which $\mathrm K_A \sslash s\mathrm{Aut}_\cI(\mathcal K,A)$ is a module, since the action of $s\mathrm{Aut}_\cI(\mathcal K,A)$ on $\mathrm K_A$ is $\mathrm{KO}$- or $\mathrm{KU}$-linear, respectively, and satisfies the Eckmann--Hilton condition of Definition~\ref{def:EH-for-R} in case of the base fields. 

 We abbreviate the homotopical twisted $K$-theory groups $(K_A \sslash \!s\mathrm{Aut}_\cI(\mathcal K,A))^i(X,\!\tau)$ to $K_A^i(X,\tau)$ and refer to~\cite[Definition~7.10]{HSS-retractive} for their definition in the setup of parametrized symmetric spectra. To compare this version of twisted $C^*$-$K$-theory to the operator algebraic one, we denote by $\Gamma(X,\xi)$ the $C^*$-algebra of sections for a bundle $\xi$ of $C^*$-algebras over a compact space $X$.

\begin{proposition}\label{prop:k-cycles}
For any $C^*$-algebra $A$, there is a canonical natural isomorphism
\[K_A^i(X,\tau) \rightarrow K_{-i}\Gamma(X,\tau^*\gamma^n_{A})\]
for $\tau \colon X \rightarrow Bs\mathrm{Aut}_\cI(\mathcal K,A)_\bld n$  with $n \geq 1$ and $X$ a finite cell complex, where 
\[\gamma^n_{A} = B^\times(*, s\mathrm{Aut}_\cI(\mathcal K,A)_\bld n, A \otimes  \mathcal K^{\otimes n})\] is the associated bundle of $C^*$-algebras. Furthermore, the diagram
\[\xymatrix@-1pc{K_A^i(X,\tau) \otimes K_B^j(Y,\sigma) \ar[r]^-\times \ar[d] & K_{A\otimes B}^{i+j}(X \times Y, \tau \otimes \sigma) \ar[d] \\
             K_{-i}(\Gamma(X,\tau^*\gamma^n_A)) \otimes K_{-j}(\Gamma(Y,\sigma^*\gamma^m_A)) \ar[r]^-\boxtimes& K_{-(i+j)}(\Gamma(X \times Y,(\tau \otimes \sigma)^*\gamma^{n+m}_{A \otimes B}))}\]
commutes, where the $\boxtimes$ on the lower horizontal arrow denotes the exterior pairing in $K$-theory applied to the pairing of $C^*$-algebras given by tensoring sections.
\end{proposition} 

Here we have abused notation by also writing $\tau$ for the composite 
\[X \xrightarrow{\tau} Bs\mathrm{Aut}_\cI(\mathcal K,A)_\bld n \rightarrow  Bs\mathrm{Aut}_\cI(\mathcal K,A)_{h\cI}\]
 of $\tau$ with the canonical map to the homotopy colimit. Using similar abuse, the groups $K_*\Gamma(X,\tau^*\gamma^k_A)$ are all canonically isomorphic once $k \geq n \geq 1$ since the pullback of $\gamma^k_A$ to $Bs\mathrm{Aut}_\cI(\mathcal K,A)_\bld n$ is isomorphic to $\gamma^n_A$ under some identification $\mathcal K^{\otimes n} \cong \mathcal K^{\otimes k}$ (of which there is a path-connected space of choices). The superscripts in the lower row of the diagram are therefore just a convenient way to formulate the isomorphism.

\begin{remark}
We shall make use of the proof expounded in \cite[Theorem 2.14 (c)]{DP-II}. The assumption in loc.cit. that $A$ be strongly selfabsorbing is not used in this part of the argument. However, the condition that $X$ be a finite cell complex (and not just compact) is missing from the statement, but clearly required, as the $K$-theory of $C^*$-algebras is not invariant under weak equivalences (even without any twists involved). The oversight occurs in \cite[Corollary 2.13]{DP-II} as the function spectrum appearing in the displayed equation needs to be derived to make the statement true. In particular, the argument for \cite[Theorem 2.14(c)]{DP-II} remains unaffected by this change.
\end{remark}

\begin{proof}[Proof of Proposition \ref{prop:k-cycles}]
For an $\cI$-space $Y$, it is shown in~\cite[Proposition~7.1]{HSS-retractive} that the functor $\mathbb R (Y \to *)^* \colon \mathrm{Ho}(\Spsym{}) \to \mathrm{Ho}(\Spsym{Y})$ has a right adjoint $\mathbb R \Gamma$ that is represented by forming section spaces in each level in case that $Y$ is a constant $\cI$-space~\cite[Lemma~7.27]{HSS-retractive}. Moreover, we write $\tau_\cI \colon X_\cI \rightarrow Bs\mathrm{Aut}_\cI(\mathcal K,A)$ for the \emph{$\cI$-spacifi\-cation} of $\tau$ discussed for example in~\cite[Proposition 4.2]{Basu_SS_Thom} and recall that by~\cite[Definition~7.10]{HSS-retractive}, we have  
\[K_A^i(X,\tau) = \pi_{-i}(\mathbb R \Gamma\mathbb R \tau_\cI^*(\mathrm{K}_A \sslash s\mathrm{Aut}_\cI(\mathcal K,A))).\]
The $\cI$-spacification comes equipped with a map $p \colon X_\cI \rightarrow \const_\cI X^{\mathrm{fib}}$ for a fibrant replacement of $(X,\tau)$ in the category $\cS/Bs\mathrm{Aut}_\cI(\mathcal K,A)_{h\cI}$. In particular, $X^{\mathrm{fib}}$ can be chosen cofibrant. Similarly, we may assume $X_\cI$ levelwise cofibrant by applying singular complexes without changing $K_A^i(X,\tau)$.
Applying the identification~\cite[(7.29)]{HSS-retractive} to $p$ and abbreviating $\mathrm{K}_A \sslash s\mathrm{Aut}_\cI(\mathcal K,A)$ by $E$ we find for any positive $m$ greater than both $i$ and $n$
\[K_A^i(X,\tau) = \pi_{-i}\mathbb R \Gamma\mathbb R \tau_\cI^*E \cong \pi_{m-i}\Gamma(X^{\mathrm{fib}}, (\mathbb Lp_! \tau_\cI^* E)^{\mathrm{fib}}_\bld m)\]
since $E$ is positive fibrant in the local model structure. 
The space occurring on the right is equivalent to 
\[\Gamma(X_\bld m,p^*(\mathbb Lp_! \tau^*_\cI E)^{\mathrm{fib}}_\bld m) = \Gamma(X_\bld m,(\mathbb Rp^*\mathbb Lp_! \tau^*_\cI E)_\bld m) \simeq \Gamma(X_\bld m,\tau^*_mE_\bld m)\]
since  $p$ is a weak equivalence and $\Gamma$ preserves weak equivalences between Serre-fibrations with cofibrant bases (confer the proof of \cite[Lemma~7.27]{HSS-retractive}). Now consider the diagram
\[\xymatrix@-1pc{X_\bld m \ar[r]^{p_\bld m}\ar[d] & X^{\mathrm{fib}}\ar[d] & X \ar[l]\ar[d]^\tau \\
\overline{Bs\mathrm{Aut}_\cI(\mathcal K,A)}_{\bld m} \ar[rd] \ar[r]& Bs\mathrm{Aut}_\cI(\mathcal K,A)_{h\cI}& \ar[l] Bs\mathrm{Aut}_\cI(\mathcal K,A)_{\bld n} \ar[ld]^{\iota}\\
& Bs\mathrm{Aut}_\cI(\mathcal K,A)_{\bld m} \ar[u]& } \]
whose composition $X_\bld m \rightarrow Bs\mathrm{Aut}_\cI(\mathcal K,A)_{\bld m}$ defines $\tau_\bld m$; here the overline on the middle left term denotes the bar resolution 
\[\overline M(n) = \mathrm{hocolim}_{k \in \cI/n} M(k)\]
occurring in the $\cI$-spacification procedure of \cite[Section 4.1]{Basu_SS_Thom}. The two rectangles are commutative and the triangles commute up to canonical homotopy (for every choice of map $\iota \colon \bld n \rightarrow \bld m$ inducing the right diagonal arrow). Furthermore, all but the upper vertical maps are equivalences. We thus obtain an identification
\[\pi_{m-i}\Gamma(X_\bld m,\tau^*_\bld mE_\bld m) \cong \pi_{m-i}\Gamma(X,\tau^*(\mathrm{K}_A \sslash s\mathrm{Aut}_\cI(\mathcal K,A))_\bld m).\]
To this object the proof of \cite[Theorem 2.7 (c)]{DP-II} applies verbatim to produce an identification with $K_{-i}\Gamma(X,(\iota\tau)^*\gamma^m_{A})$ and by the explanation following the statement of the proposition this is canonically identified with $K_{-i}\Gamma(X,\tau^*\gamma^n_{A})$, independent of the choice of $\iota$.

The second claim is a lengthy, but straightforward diagram chase using the monoidal structure for $\mathbb R \Gamma$ described in~\cite[(7.28)]{HSS-retractive}.
\end{proof}

We assume from here on that $A$ is unital. Then we can also consider the following variant of the spectrum above, see \cite[Definition 4.1]{DP-I}:
	\[(\mathrm K^\infty_A)_\bld n = \mathrm{Hom}(\widehat{\mathrm S}, A^{\otimes n} \otimes \mathcal K(L^2(\mathbb R^n)) \otimes \mathrm{Cl}_n)\]
This spectrum is in fact a commutative symmetric ring spectrum using the same multiplication as above and is acted on by $\mathrm{Aut}^s_{\cI}(A \otimes \cK)$ via post composition. In generalization of the statement for the base fields, the action is easily checked to satisfy the Eckmann--Hilton condition of Definition~\ref{def:EH-for-R}. Let us quickly point out, that $\mathrm K_A^\infty$ does \emph{not} usually model the $K$-theory spectrum $\mathrm{K}_A$ of $A$. 
Rather its homotopy groups are related to the $K$-theory of `$A^{\otimes \infty}$' (though we shall not attempt to make this precise). There is always the comparison zig-zag
\begin{equation}\label{K-comparison}
\mathrm K_{A} \longrightarrow K^{\infty+1}_A \longleftarrow K_A^\infty
\end{equation}
with
\[(\mathrm K^{\infty+1}_A)_\bld n = \mathrm{Hom}(\widehat{\mathrm S}, A^{\otimes n+1} \otimes \mathcal K(L^2(\mathbb R^{n+1})) \otimes \mathrm{Cl}_n),\]
which was denoted $\mathrm{KU}_\bullet^{A,mod}$ in \cite{DP-II}; the left map is given by tensoring with the identity in the remaining factors, while the right hand map depends on a rank one projection in $\mathcal K$ (confer the proof of \cite[Theorem 2.14~(b)]{DP-II}). Unfortunately, it does not seem to be known whether in general the right hand map is a stable equivalence (it is not usually a $\pi_\ast$-isomorphism). 

The case of greatest interest for this variant of $K$-theory spectra is that of strongly selfabsorbing $C^*$-algebras, for which there exists an isomorphism $A \rightarrow A \otimes A$ that is homotopic to either inclusion, see \cite{TW-selfabs} for a precise definition. In this case it is easy to see that \emph{both} maps are $\pi_*$-isomorphisms (see again \cite[Theorem~2.14~(b)]{DP-II}). So in this case $\mathrm K^\infty_A$ does model the $K$-theory of $A$, and is in fact a positive $\Omega$-spectrum as well; the multiplication is quite different in flavor from that on the $K$-theory spectrum $\mathrm K_A$ if A is commutative though (except when $A$ is the base field, in which case both comparison maps are isomorphisms of ring spectra). Investing the main theorem of \cite{DP-I} we find a multiplicative version of \cite[Theorem 2.14 (c)]{DP-II}:

\begin{corollary}\label{tata}
For a strongly selfabsorbing, purely infinite $C^*$-algebra $A$ there is a zig-zag of local equivalences \[\gamma_{\mathrm K_A^\infty} \simeq \mathrm K_A^\infty \sslash \mathrm{Aut}^s_{\cI}(A \otimes \cK)\]
of commutative parametrized ring spectra. Furthermore, for any map $\tau \colon X \rightarrow B\mathrm{Aut}(A \otimes \mathcal K)$, there is a canonical isomorphism $(K^\infty_A)^*(X,\tau) \rightarrow (K_A)^*(X,\tau)$
and for any finite cell complex $X$ the resulting diagram
\[\xymatrix@-1pc{K_A^i(X,\tau) \otimes K_A^i(Y,\sigma) \ar[r]^-\times \ar[d] & K_{A}^{i+j}(X \times Y, \tau + \sigma) \ar[d] \\
             K_{-i}(\Gamma(X,\tau^*\gamma^1_A)) \otimes K_{-i}(\Gamma(Y,\sigma^*\gamma^1_A)) \ar[r]^-\boxtimes& K_i(\Gamma(X \times Y,(\tau \otimes \sigma)^*\gamma^{1}_{A}))}\]
commutes, where we have used a witnessing isomorphism of the strong selfabsorption to identify the lower right hand corner. 
\end{corollary}

For the notion of purely infinite $C^*$-algebras we refer the reader to \cite[Proposition 1.6]{Cuntz}. Let us only remark that the base field is rightly excluded from being purely infinite.

Specializing to the infinite Cuntz-algebra $\mathcal O_\infty$, whose unit induces an isomorphism $\mathrm{KU} \rightarrow \mathrm K_{\mathcal O_\infty}$ by \cite[Corollary 3.11]{Cuntz}, we find:

\begin{corollary}\label{tata2}
There is a zig-zag of local equivalences
\[\gamma_{\mathrm{KU}} \simeq \mathrm K_{\mathcal O_\infty}^\infty \sslash \mathrm{Aut}^s_{\cI}({\mathcal O_\infty} \otimes \cK)\]
of commutative parametrized ring spectra, lifting Pennig's description of twisted $K$-theory to a multiplicative isomorphism.
\end{corollary}

\begin{proof}[Proof of Corollary \ref{tata}]
The first claim follows immediately from Corollary \ref{comparison corollary}, \cite[Theorem 1.1]{DP-I}, which shows that $\mu^\flat \colon \mathrm{Aut}_\cI^s(A \otimes \mathcal K) \rightarrow \GLoneIof(\mathrm K_A^\infty)$ is an $\cI$-equivalence and \cite[Proposition 2.26]{DP-III}, which shows that $\mathrm{Aut}^s_\cI(A \otimes \mathcal K)$ is well-pointed and therefore does not need to be replaced by its simplicial counterpart.

To obtain the other claims consider the comparison zig-zag \eqref{K-comparison}. We find a commutative diagram
\[\xymatrix@-1pc{\mathrm K_{A} \sm \mathrm K_{A} \ar[r] \ar[d] & K^{\infty+1}_A \sm K^{\infty+1}_A \ar[d] & K_A^\infty \sm K_A^\infty \ar[l] \ar[d] \\
\mathrm K_{A \otimes A} \ar[r] & K^{\infty+2}_A & K^{\infty}_A \ar[l]}\]
where the notation is supposed to be self-explanatory. Now the $\cI$-monoid $\mathrm{Aut}_\mathcal I(\mathcal K,A)$ acts on $\mathrm K_A$, $\mathrm{Aut}_\mathcal I(\mathcal K, A \otimes A)$ acts on $\mathrm K_{A \otimes A}$, $\mathrm{Aut}_\mathcal I(A \otimes \mathcal K, A \otimes \mathcal K)$ acts on $K^{\infty+1}_A$, $\mathrm{Aut}_\mathcal I(A \otimes \mathcal K, A^{\otimes 2} \otimes \mathcal K^{\otimes 2})$ acts on $K^{\infty+2}_A$ and finally $\mathrm{Aut}^s_{\cI}(A \otimes \cK)$ on $K^{\infty}_A$. These actions are compatible with the various induced maps and by (the proof of) \cite[Theorem 4.5]{DP-I} all inclusions between them are levelwise homotopy equivalences. Therefore applying $-\sslash-$ to every available term produces the desired zig-zag (even of positive level equivalences) and identifies the internal pairing of the right hand side with the external one on the left.

The claim about the diagram commuting is now immediate from Proposition~\ref{prop:k-cycles}.
\end{proof}

\begin{remark}
It is not clear to us what the analogous statements for the associated homology theory should be in the full generality of Proposition \ref{prop:k-cycles}. In case $\tau$ factors through $Bs\mathrm{Aut}_\cI(\mathcal K)_\bld n \rightarrow Bs\mathrm{Aut}_\cI(\mathcal K,A)_\bld n$, there are multiplicative isomorphisms
\[(K_A)_i(X,\tau) \cong \mathrm{KK}_i(\Gamma((-\tau)^*\gamma^n_\mathbb K),A)\]
as a straightforward generalization of the untwisted case, where 
\[(K_A)_i(X) \cong \mathrm{KK}_i(C(X),A) \cong \mathrm{KK}_i(C(X, \mathcal K),A).\]
In case $A$ is strongly selfabsorbing, and also for the base field $A = \mathbb K$, \cite[Theorem 3.18]{DP-III} implies that
\[(K_A)_i(X,\tau) \cong \mathrm{KK}_i(\Gamma((-\tau)^*\gamma^n_A),A)\]
for arbitrary $\tau \colon X \rightarrow Bs\mathrm{Aut}_\cI(\mathcal K,A)_\bld n$ by a similar reduction as in the proof of Proposition \ref{prop:k-cycles}. Again we expect the isomorphism to be multiplicative. 

However, we do not know of a common generalization of these two descriptions and refrain from spelling them out, as the construction of the isomorphism would require a lengthy detour into $\mathrm{KK}$-theory.
\end{remark}

\subsection{Comparison of geometric and tautological twists of  \texorpdfstring{$K$}{K}-theory}\label{comp-geo-homo} 
We will now use the results above to obtain the comparison between the operator algebraic and homotopically minded spectra representing twisted $K$-theory. As an additional input, we use the Atiyah--Bott--Shapiro orientations
	\[\alpha \colon \mathrm{MSpin} \longrightarrow \mathrm{KO}	\quad \text{and} \quad \alpha^c \colon \mathrm{MSpin}^c \longrightarrow \mathrm{KU}\]
	from \cite[Section 6]{Jo-coherence}.  Recall first that the maps $\alpha$ and $\alpha^c$ are obtained from the unit maps $S^n \rightarrow \mathrm {KO}_n$ or $S^n \rightarrow \mathrm{KU}_n$ (see \cite[Construction 3.4.1]{HeJo}), by extending them over $\mathrm{MSpin}_n = (\mathrm{P}\widehat{\mathcal O}_n)_+ \wedge_{\mathrm O(n)} S^n$ and its complex analogue in the unique $\mathrm{P}\widehat{\mathcal O}_n$- or $\mathrm{P}\widehat{\mathcal U}_n$-equivariant fashion. Also recall that $\mathrm{Aut}_\cI(\mathcal K) = P\mathcal O_\cI$ acts on $\mathrm{KO}$ (and that $P_{S^1}\mathcal U_\cI$ acts on $\mathrm{KU}$) by the constructions of the previous section. This allows us to implement the Donovan--Karoubi map considered in the introduction in the present setup as the map
%        \begin{equation}\label{eq:DK-map}
\[
  \kappa \colon B^{\boxtimes}(\mathrm{P}{\mathcal O}_\cI^{\mathrm{cof}}) \to \BGLoneIcofof{\mathrm{KO}}
  \]
%  \end{equation}
  induced by~\eqref{eq:mu-sharp}, and similarly in the complex case. On the other hand, composing 
  the map of classifying spaces induced by the map $\taut_{\mathrm{Spin}}$ from the construction in Section~\ref{section:taut-twists} with the Atiyah--Bott--Shapiro orientation provides a map
  \[ \tautko\colon B^{\boxtimes}((\mathrm{O}_\cI \sslash \mathrm{Spin}_\cI)^{\mathrm{cof}}) \to  \BGLoneIcofof{\mathrm{KO}} \]
  that implements the inclusion of tautological twists of $K$-theory $\tautko$ in our setup. 
We obtain the identification of $\kappa$ first sketched in \cite[Appendix~C]{HeJo}:
  \begin{theorem}
  The Donovan--Karoubi map $\kappa$  and the inclusion of tautological twists of $K$-theory $\tautko$ are related by a zig-zag of $\cI$-equivalences between commutative $\cI$-space monoids over $\BGLoneIcofof{\mathrm{KO}}$, and analogously in the complex case. In particular, this zig-zag induces an isomorphism between these two maps in the homotopy category of $E_{\infty}$ spaces. 
  \end{theorem}
  \begin{proof}
   Recall that $B^{\boxtimes}(\mathrm{proj}^\ttt) \colon B^{\boxtimes}(\mathrm{P}{\mathcal O}_\cI^{\mathrm{cof}}) \rightarrow \BGLoneIof{(M\mathrm{Spin})^{\mathrm{cof}}}$ is the map associated to the evident action of $\mathrm{P}{\mathcal O}_\cI$ on $\mathrm{MSpin} = \Theta(\widehat{\mathcal O}^{\mathrm{ev}}_\cI \times_{\mathrm{Spin}_\cI} \bS)$.   From \eqref{taut vs DK} we obtain a zig-zag of $\cI$-equivalences between the maps of commutative $\cI$-space monoids $B^{\boxtimes}(\taut_{\mathrm{Spin}}^\mathrm{cof})$ and $B^{\boxtimes}((\mathrm{proj}^\ttt)^\mathrm{cof})$.    
  Since the Atiyah--Bott--Shapiro orientation $\alpha \colon \mathrm{MSpin} \to\mathrm{KO}$ is equivariant for the $P\mathcal O_\cI$-actions, $\kappa$ factors as the composite \[ B((\mathrm{P}{\mathcal O}_\cI)^{\mathrm{cof}}) \xrightarrow{B^{\boxtimes}((\mathrm{proj}^{\ttt})^\mathrm{cof})} \BGLoneIcofof{\mathrm{MSpin}}\xrightarrow{\BGLoneIof{(\alpha)}} \BGLoneIcofof{\mathrm{KO}}.\] 
This implies the claim. 
\end{proof}

We can now prove the identification $(\tautko)^*\gamma_{\mathrm{KO}} \simeq \mathrm{KO} \sslash \mathrm P \mathcal O$ stated in the introduction:
\begin{proof}[Proof of Theorem~\ref{thm:d}]
Corollary~\ref{comparison corollary} provides a zig-zag of local equivalences between commutative parametrized ring spectra relating $\mathrm{KO} \sslash \mathrm P\mathcal O_\cI$ and $\kappa^* \gamma_{\mathrm{KO}}$. The zig-zag of $\cI$-equivalences between $\kappa$ and $\tautko$ established in the previous theorem implies that the parametrized spectra obtained from pulling back along these maps are locally equivalent. This identifies $\kappa^*\gamma_{\mathrm{KO}}$  with $(\tautko)^*\gamma_{\mathrm{KO}}$.   \end{proof}
Combining Theorem~\ref{thm:d} with Proposition \ref{prop:k-cycles} provides the desired multiplicative description of the twisted (co)homology theory associated to $\gamma_\mathrm{KO}$ in terms of $\mathrm{KK}$-theory.

\begin{remark}
To the best of our knowledge such a comparison of operator algebraic and homotopical twisted K-theory has not appeared in the literature before, although it has been variously used, e.g. \cite[Section 6]{Ando-B-G_twists} and \cite[Section 1.1.1 and 1.1.2]{douglas-twist}. As mentioned in the introduction its rudiments were attempted in \cite{AGG-Uniqueness}, where it was shown that the maps \[\mu^\flat \colon B\mathrm{P}{\mathcal O} \rightarrow \BGLoneof{\mathrm{KO}}\quad \text{ and } \quad \BGLoneof{(\alpha)} \circ \taut_{\mathrm{Spin}} \colon B(\mathrm O \sslash \mathrm{Spin}) \rightarrow \BGLoneof{\mathrm{KO}}\] and their complex analogues give homotopic maps on connective covers, i.e., that their images in
\[[K(\bZ/2,2),\BGLoneof{\mathrm{KO}}] \quad \text{and} \quad [K(\bZ,3),\BGLoneof{\mathrm{KU}}]\]
agree, by explicitly determining these groups to be $\bZ/2$ and $\bZ$, respectively. By contrast, our result compares the two maps at the level of Spin bordism spectra, as $\mathrm{E}_\infty$-maps, before taking connective covers and allows for the direct interpretation of homotopically twisted $K$-theory via operator algebras.\\

We also immediately obtain maps
\[\alpha \colon \mathrm{MSpin}\sslash \mathrm P\mathcal O_\cI \longrightarrow \mathrm{KO}\sslash \mathrm P\mathcal O_\cI \quad \text{and} \quad \alpha^c \colon \mathrm{MSpin}^c\sslash P_{S^1}\mathcal U_\cI \longrightarrow \mathrm{KU}\sslash P_{S^1}\mathcal U_\cI,\]
of commutative parametrized ring spectra, geometric incarnations of the  \emph{twisted Atiyah--Bott--Shapiro orientations}. These maps are refinements of the analogous map considered in \cite{HeJo} based on the framework for parametrized homotopy theory of May and Sigurdsson \cite{May-S_parametrized}, which did not allow for multiplicative considerations. By our results above, the left one fits into a commutative square
\[\xymatrix@-1pc{\mathrm{MSpin} \sslash \mathrm P\mathcal O_\cI \ar[rr]^\alpha && \mathrm{KO}\sslash \mathrm P\mathcal O_\cI \\
(B^{\boxtimes}(\taut_{\mathrm{Spin}}^\mathrm{cof}))^*\gamma_{M\mathrm{Spin}} \ar[rr]^-{\gamma_\alpha} \ar[u] && (\tautko)^*\gamma_\mathrm{KO} \ar[u]^\simeq}\]
in the homotopy category of commutative parametrized ring spectra, and analogously for the right one. Given the (untwisted) Atiyah-Bott-Shapiro orientation, the construction of the lower map is entirely homotopical in nature and is also easily carried out in other frameworks for parametrized homotopy theory such as \cite{Ando-B-G-H-R_infinity-Thom}. Theorem \ref{comparison theorem} does not, however, suffice to see that the left vertical map is an isomorphism, since $\mathrm{MSpin}$ is not positively fibrant as a symmetric ring spectrum. As mentioned in Remark~\ref{rem:action-of-F}, we expect a semistability assumption to suffice for its conclusion instead of fibrancy, but the future applications we have in mind are not affected by replacing $\mathrm{MSpin}$ fibrantly. 
\end{remark}
%\bibliography{retractive}

\begin{bibdiv}
\begin{biblist}

\bib{Ando-B-G_twists}{incollection}{
      author={Ando, Matthew},
      author={Blumberg, Andrew~J.},
      author={Gepner, David},
       title={Twists of {K}-theory and {TMF}},
        date={2010},
   booktitle={Superstrings, geometry, topology, and {$C^\ast$}-algebras},
      volume={81},
   publisher={Amer. Math. Soc., Providence, RI},
       pages={27\ndash 63},
         url={http://dx.doi.org/10.1090/pspum/081/2681757},
        note={Proc. Sympos. Pure Math., Amer. Math. Soc., Providence, RI, Vol.
  81},
}

\bib{Ando-B-G_parametrized}{article}{
      author={Ando, Matthew},
      author={Blumberg, Andrew~J.},
      author={Gepner, David},
       title={Parametrized spectra, multiplicative {T}hom spectra, and the
  twisted {U}mkehr map},
        date={2018},
     journal={Geom. Topol.},
      volume={22},
       pages={3761\ndash 3825},
}

\bib{Ando-B-G-H-R_infinity-Thom}{article}{
      author={Ando, Matthew},
      author={Blumberg, Andrew~J.},
      author={Gepner, David},
      author={Hopkins, Michael~J.},
      author={Rezk, Charles},
       title={An {$\infty$}-categorical approach to {$R$}-line bundles,
  {$R$}-module {T}hom spectra, and twisted {$R$}-homology},
        date={2014},
        ISSN={1753-8416},
     journal={J. Topol.},
      volume={7},
      number={3},
       pages={869\ndash 893},
         url={http://dx.doi.org/10.1112/jtopol/jtt035},
}

\bib{AGG-Uniqueness}{article}{
      author={Antieau, Benjamin},
      author={Gepner, David},
      author={G\'omez, Jos\'e~Manuel},
       title={Actions of {$K(\pi,n)$} spaces on {$K$}-theory and uniqueness of
  twisted {$K$}-theory},
        date={2014},
        ISSN={0002-9947},
     journal={Trans. Amer. Math. Soc.},
      volume={366},
      number={7},
       pages={3631\ndash 3648},
         url={https://doi.org/10.1090/S0002-9947-2014-05937-0},
}

\bib{supercat}{article}{
      author={Brundan, Jonathan},
      author={Ellis, Alexander~P.},
       title={Monoidal supercategories},
        date={2017},
        ISSN={0010-3616},
     journal={Comm. Math. Phys.},
      volume={351},
      number={3},
       pages={1045\ndash 1089},
         url={https://doi.org/10.1007/s00220-017-2850-9},
}

\bib{Basu_SS_Thom}{article}{
      author={Basu, Samik},
      author={Sagave, Steffen},
      author={Schlichtkrull, Christian},
       title={{G}eneralized {T}hom spectra and their topological {H}ochschild
  homology},
        date={2020},
        ISSN={1474-7480},
     journal={J. Inst. Math. Jussieu},
      volume={19},
      number={1},
       pages={21\ndash 64},
         url={https://doi.org/10.1017/s1474748017000421},
}

\bib{Cuntz}{article}{
      author={Cuntz, Joachim},
       title={{$K$}-theory for certain {$C^{\ast} $}-algebras},
        date={1981},
        ISSN={0003-486X},
     journal={Ann. of Math. (2)},
      volume={113},
      number={1},
       pages={181\ndash 197},
         url={https://doi.org/10.2307/1971137},
}

\bib{Dixmier-book}{book}{
      author={Dixmier, Jacques},
       title={{$C^*$}-algebras},
   publisher={North-Holland Publishing Co., Amsterdam-New York-Oxford},
        date={1977},
        ISBN={0-7204-0762-1},
        note={Translated from the French by Francis Jellett, North-Holland
  Mathematical Library, Vol. 15},
}

\bib{Donovan}{article}{
      author={Donovan, P.},
      author={Karoubi, M.},
       title={Graded {B}rauer groups and {$K$}-theory with local coefficients},
        date={1970},
        ISSN={0073-8301},
     journal={Inst. Hautes \'{E}tudes Sci. Publ. Math.},
      number={38},
       pages={5\ndash 25},
         url={http://www.numdam.org/item?id=PMIHES_1970__38__5_0},
}

\bib{douglas-twist}{article}{
      author={Douglas, Christopher~L.},
       title={On the twisted {$K$}-homology of simple {L}ie groups},
        date={2006},
        ISSN={0040-9383},
     journal={Topology},
      volume={45},
      number={6},
       pages={955\ndash 988},
         url={https://doi.org/10.1016/j.top.2006.06.007},
}

\bib{DP-I}{article}{
      author={Dadarlat, Marius},
      author={Pennig, Ulrich},
       title={Unit spectra of {K}-theory from strongly self-absorbing
  {$C^*$}-algebras},
        date={2015},
        ISSN={1472-2747},
     journal={Algebr. Geom. Topol.},
      volume={15},
      number={1},
       pages={137\ndash 168},
         url={https://doi.org/10.2140/agt.2015.15.137},
}

\bib{DP-III}{article}{
      author={Dadarlat, Marius},
      author={Pennig, Ulrich},
       title={A {D}ixmier-{D}ouady theory for strongly self-absorbing
  {$C^*$}-algebras},
        date={2016},
        ISSN={0075-4102},
     journal={J. Reine Angew. Math.},
      volume={718},
       pages={153\ndash 181},
         url={https://doi.org/10.1515/crelle-2014-0044},
}

\bib{Dugger_replacing}{article}{
      author={Dugger, Daniel},
       title={Replacing model categories with simplicial ones},
        date={2001},
        ISSN={0002-9947},
     journal={Trans. Amer. Math. Soc.},
      volume={353},
      number={12},
       pages={5003\ndash 5027},
         url={https://doi.org/10.1090/S0002-9947-01-02661-7},
}

\bib{HeJo}{article}{
      author={Hebestreit, Fabian},
      author={Joachim, Michael},
       title={Twisted spin cobordism and positive scalar curvature},
        date={2020},
        ISSN={1753-8416},
     journal={J. Topol.},
      volume={13},
      number={1},
       pages={1\ndash 58},
         url={https://doi.org/10.1112/topo.12122},
}

\bib{Hovey_symmetric-general}{article}{
      author={Hovey, Mark},
       title={Spectra and symmetric spectra in general model categories},
        date={2001},
        ISSN={0022-4049},
     journal={J. Pure Appl. Algebra},
      volume={165},
      number={1},
       pages={63\ndash 127},
         url={http://dx.doi.org/10.1016/S0022-4049(00)00172-9},
}

\bib{HSS-retractive}{article}{
      author={Hebestreit, Fabian},
      author={Sagave, Steffen},
      author={Schlichtkrull, Christian},
       title={Multiplicative parametrized homotopy theory via symmetric spectra
  in retractive spaces},
        date={2020},
     journal={Forum Math. Sigma},
      volume={8},
       pages={e16, 84 pp.},
         url={https://doi.org/10.1017/fms.2020.11},
}

\bib{Jo-kspectra}{article}{
      author={Joachim, Michael},
       title={{$K$}-homology of {$C^\ast$}-categories and symmetric spectra
  representing {$K$}-homology},
        date={2003},
        ISSN={0025-5831},
     journal={Math. Ann.},
      volume={327},
      number={4},
       pages={641\ndash 670},
         url={https://doi.org/10.1007/s00208-003-0426-9},
}

\bib{Jo-coherence}{incollection}{
      author={Joachim, Michael},
       title={Higher coherences for equivariant {$K$}-theory},
        date={2004},
   booktitle={Structured ring spectra},
      series={London Math. Soc. Lecture Note Ser.},
      volume={315},
   publisher={Cambridge Univ. Press, Cambridge},
       pages={87\ndash 114},
         url={https://doi.org/10.1017/CBO9780511529955.006},
        note={London Math. Soc. Lecture Note Ser. \textbf{315}, Cambridge Univ.
  Press},
}

\bib{May-S_parametrized}{book}{
      author={May, J.~P.},
      author={Sigurdsson, J.},
       title={Parametrized homotopy theory},
      series={Mathematical Surveys and Monographs},
   publisher={American Mathematical Society, Providence, RI},
        date={2006},
      volume={132},
        ISBN={978-0-8218-3922-5; 0-8218-3922-5},
         url={http://dx.doi.org/10.1090/surv/132},
}

\bib{DP-II}{article}{
      author={Pennig, Ulrich},
       title={A non-commutative model for higher twisted {$K$}-theory},
        date={2016},
        ISSN={1753-8416},
     journal={J. Topol.},
      volume={9},
      number={1},
       pages={27\ndash 50},
         url={https://doi.org/10.1112/jtopol/jtv033},
}

\bib{RoSt}{incollection}{
      author={Rosenberg, Jonathan},
      author={Stolz, Stephan},
       title={Metrics of positive scalar curvature and connections with
  surgery},
        date={2001},
   booktitle={Surveys on surgery theory, {V}ol. 2},
      series={Ann. of Math. Stud.},
      volume={149},
   publisher={Princeton Univ. Press, Princeton, NJ},
       pages={353\ndash 386},
}

\bib{Schlichtkrull_units}{article}{
      author={Schlichtkrull, Christian},
       title={Units of ring spectra and their traces in algebraic
  {$K$}-theory},
        date={2004},
        ISSN={1465-3060},
     journal={Geom. Topol.},
      volume={8},
       pages={645\ndash 673 (electronic)},
}

\bib{Schlichtkrull_Thom-symmetric}{article}{
      author={Schlichtkrull, Christian},
       title={Thom spectra that are symmetric spectra},
        date={2009},
     journal={Doc. Math.},
      volume={14},
       pages={699\ndash 748},
}

\bib{schwede-global}{book}{
      author={Schwede, Stefan},
       title={Global homotopy theory},
      series={New Mathematical Monographs},
   publisher={Cambridge University Press, Cambridge},
        date={2018},
      volume={34},
        ISBN={978-1-108-42581-0},
         url={https://doi.org/10.1017/9781108349161},
}

\bib{Sagave-S_diagram}{article}{
      author={Sagave, Steffen},
      author={Schlichtkrull, Christian},
       title={Diagram spaces and symmetric spectra},
        date={2012},
        ISSN={0001-8708},
     journal={Adv. Math.},
      volume={231},
      number={3-4},
       pages={2116\ndash 2193},
         url={http://dx.doi.org/10.1016/j.aim.2012.07.013},
}

\bib{Sagave-S_group-compl}{article}{
      author={Sagave, Steffen},
      author={Schlichtkrull, Christian},
       title={Group completion and units in {$\mathcal I$}-spaces},
        date={2013},
        ISSN={1472-2747},
     journal={Algebr. Geom. Topol.},
      volume={13},
      number={2},
       pages={625\ndash 686},
         url={http://dx.doi.org/10.2140/agt.2013.13.625},
}

\bib{Steenrod}{article}{
      author={Steenrod, N.~E.},
       title={Homology with local coefficients},
        date={1943},
        ISSN={0003-486X},
     journal={Ann. of Math. (2)},
      volume={44},
       pages={610\ndash 627},
         url={https://doi.org/10.2307/1969099},
}

\bib{Stolz-conc}{misc}{
      author={Stolz, Stephan},
       title={Concordance classes of positive scalar curvature metrics},
        date={1996},
        note={Preprint, available at
  \url{https://www3.nd.edu/~stolz/preprint.html}},
}

\bib{tDKaPu70}{book}{
      author={tom Dieck, Tammo},
      author={Kamps, Klaus~Heiner},
      author={Puppe, Dieter},
       title={Homotopietheorie},
      series={Lecture Notes in Mathematics, Vol. 157},
   publisher={Springer-Verlag, Berlin-New York},
        date={1970},
}

\bib{TW-selfabs}{article}{
      author={Toms, Andrew~S.},
      author={Winter, Wilhelm},
       title={Strongly self-absorbing {$C^*$}-algebras},
        date={2007},
        ISSN={0002-9947},
     journal={Trans. Amer. Math. Soc.},
      volume={359},
      number={8},
       pages={3999\ndash 4029},
         url={https://doi.org/10.1090/S0002-9947-07-04173-6},
}

\end{biblist}
\end{bibdiv}

\end{document}